\documentclass[10pt,a4paper]{article}

\usepackage{amsmath,amsthm,amsbsy}
\usepackage{amsfonts}
\usepackage{amssymb}
\usepackage{graphicx}
\usepackage{caption}
\usepackage{subcaption}

\usepackage[ulem=normalem]{changes}
\usepackage{tikz}
\usetikzlibrary{decorations.pathreplacing}

\newtheorem{definition}{Definition}

\newtheorem{theorem}{Theorem}
\newtheorem{lemma}{Lemma}
\newtheorem{corollary}{Corollary}

\theoremstyle{definition}
\newtheorem{ex}{Example}

\def\e{\mathop{\boldsymbol e}\nolimits}
\def\t{\mathop{\boldsymbol t}\nolimits}
\def\u{\mathop{\boldsymbol u}\nolimits}
\def\v{\mathop{\boldsymbol v}\nolimits}
\def\w{\mathop{\boldsymbol w}\nolimits}

\def\T{\mathop{\boldsymbol T}\nolimits}
\def\Tm{\mathop{\boldsymbol T}_{\!m+2}\nolimits}
\def\Tt{\mathop{\boldsymbol T}_{\!3}\nolimits}
\def\TT{\mathop{\boldsymbol T}_{\!4}\nolimits}

\def\O{\mathop{\mathcal O}\nolimits}
\def\C{\mathop{\mathcal C}\nolimits}
\def\P{\mathop{\mathcal P}\nolimits}

\begin{document}

\title{Limiting Behaviour of Fr\'echet Means \\ in the Space of Phylogenetic Trees}

\author{D. Barden\thanks{Girton College, University of Cambridge, Cambridge, CB3 0JG, UK ({\tt d.barden@dpmms.cam.ac.uk}).} \and H. Le\thanks{School of Mathematical Sciences, University of Nottingham, Nottingham, NG7 2RD, UK ({\tt huiling.le@nottingham.ac.uk}).} \and M. Owen\thanks{Department of Mathematics and Computer Science, Lehman College, City University of New York, Bronx, 10468, USA
({\tt megan.owen@lehman.cuny.edu}).}} 

\date{}%\today}

\maketitle

\begin{abstract}
As demonstrated in our previous work on $\TT$, the space of phylogenetic trees with four leaves, the global, as well as the local, topological structure of the space plays an important role in the non-classical limiting behaviour of the sample Fr\'echet means of a probability distribution on $\TT$. Nevertheless, the techniques used in that paper were specific to $\TT$ and cannot be adapted to analyse Fr\'echet means in the space ${\boldsymbol T}_{\!m}$ of phylogenetic trees with $m(\geqslant5)$ leaves. To investigate the latter, this paper first studies the log map of ${\boldsymbol T}_{\!m}$, a generalisation of the inverse of the exponential map on a Riemannian manifold. Then, in terms of a modified version of the log map, we characterise Fr\'echet means in ${\boldsymbol T}_{\!m}$ that lie in top-dimensional or co-dimension one strata. We derive the limiting distributions for the corresponding sample Fr\'echet means, generalising our previous results. In particular, the results show that, although they are related to the Gaussian distribution, the forms taken by the limiting distributions depend on the co-dimensions of the strata in which the Fr\'echet means lie.
\end{abstract}



\section{Introduction}

The concept of Fr\'echet means of random variables on a metric space is a generalisation of the least mean-square characterisation of Euclidean means: a point is a Fr\'echet mean of a probability measure $\mu$ on a metric space $(\boldsymbol{M},d)$ if it minimises the Fr\'echet function for $\mu$ defined by
\begin{eqnarray*}
x\mapsto\frac{1}{2}\int_{\boldsymbol M}d(x,x')^2d\mu(x'),
\end{eqnarray*}
provided the integral on the right side is finite for at least one point $x$. Note that the factor $1/2$ will simplify some later computations. The concept of Fr\'echet means has recently been used in the statistical analysis of data of a non-Euclidean nature. We refer readers to \cite{BP}, \cite{BP2}, \cite{DLPW}, \cite{DM} and \cite{KL}, as well as the references therein, for the relevance of, and recent developments in, the study of various aspects of Fr\'echet means in Riemannian manifolds. The Fr\'echet mean has also been studied in the space of phylogenetic trees, as motivated by \cite{BHV} and \cite{SH}. It was first introduced to this space independently by \cite{MB} and \cite{MOP}, which both gave methods for computing it. Limiting distributions of sample Fr\'echet means in the space of phylogenetic trees with four leaves were studied in \cite{BLO}, and it was used to analyse tree-shaped medical imaging data in \cite{FOPWTDB}, while principal geodesic analysis on the space of phylogenetic trees, a related statistical issue, was studied in \cite{TN}, \cite{TN2}, and \cite{FOPWTDB}.

A phylogenetic tree represents the evolutionary history of a set of organisms and is an important concept in evolutionary biology. Such a tree is a contractible graph, that is, a connected graph with no circuits, where one of its vertices of degree 1 is distinguished as the root of the tree and the other such vertices are (labelled) leaves. The space ${\boldsymbol T}_{\!m}$ of phylogenetic trees with $m$ leaves was first introduced in \cite{BHV}. The important feature of the space is that each point represents a tree with a particular structure and specified lengths of its edges in such a way that both the structure and the edge lengths vary continuously in a natural way throughout the space. The space is constructed by identifying faces of a disjoint union of Euclidean orthants, each corresponding to a different tree structure.  In particular, it is a topologically stratified space and also a $CAT(0)$, or globally non-positively curved, space (cf. \cite{BH}). A detailed account of the underlying geometry of tree spaces can be found in \cite{BHV} and a brief summary can be found in the Appendix to \cite{BLO}.

As demonstrated in \cite{BB} and \cite{HHLMMMNOPS} for $\Tt$ and in \cite{BLO} for $\TT$, the global, as well as the local, topological structure of the space of phylogenetic trees plays an important role in the limiting behaviour of sample Fr\'echet means. These results imply that the known results (cf. \cite{KL}) on the limiting behaviour of sample Fr\'echet means in Riemannian manifolds cannot be applied directly. Moreover, due to the increasing complexity of the structure of ${\boldsymbol T}_{\!m}$ as $m$ increases, the techniques used in \cite{BLO} for $\TT$ could not be adapted to derive the limiting behaviour of sample Fr\'echet means in ${\boldsymbol T}_{\!m}$ for general $m$. For example, although the natural isometric embedding of $\TT$ is in 10-dimensional Euclidean space $\mathbb{R}^{10}$, it is intrinsically 2-dimensional, being constructed from 15 quadrants identified three at a time along their common axes. This made it possible in \cite{BLO}, following \cite{BHV}, to represent $\TT$ as a union of certain quadrants embedded in $\mathbb{R}^3$ in such a way that it was possible to visualise the geodesics explicitly. That is, naturally, not possible for $m>4$. The need to describe geodesics explicitly arises as follows. In a complete manifold of non-positive curvature, the global minimum of a Fr\'echet function would be characterised by the vanishing of its derivative. In tree space, as in general stratified spaces, such derivatives do not exist at non-manifold points. However directional derivatives for a Fr\'echet function, which serve our purpose, do exist at all points and for all tangential directions. They are defined via the log map, which is a generalisation of the inverse of the exponential map of Riemannian manifolds, and is expressed in terms of the lengths and initial tangent vectors of unit speed geodesics. In this paper, we derive these data using the geometric structure of geodesics in ${\boldsymbol T}_{\!m}$ obtained in \cite{MO} and \cite{OP}. As a result, we are able to establish a central limit theorem for \textit{iid} random variables having probability measure $\mu$ that has its Fr\'echet mean lying in a top-dimensional stratum. We are also able to take advantage of the special structure of tree space in the neighbourhood of a stratum of co-dimension one to obtain the analogous results when the Fr\'echet mean of $\mu$ lies in such a stratum. For the latter case, we show that the limiting distribution can take one of three possible forms, distinguished by the nature of its support. Unlike the Euclidean case, the limiting distributions in both cases here are expressed in terms of the log map at the Fr\'echet mean of $\mu$. This is similar to the central limit theorem for sample Fr\'echet means on Riemannian manifolds (cf. \cite{KL}). Although it may appear non-intuitive, it allows us to use the standard results on Euclidean space. For example, in the top-dimensional case, the limiting distribution is a Gaussian distribution and so some classical hypothesis tests can be carried out in a similar fashion to hypothesis tests for data lying in a Riemannian manifold as demonstrated in \cite{DM} for the statistical analysis of shape. However, in the case of co-dimension one, the limiting distribution is non-standard and so the classical hypothesis tests can not easily be modified to apply. Further investigation is required and we aim to pursue this, as well as the applications of the results to phylogenetic trees, in future papers.

\vskip 6pt
The remainder of the paper is organised as follows. In order to obtain the directional derivatives of a Fr\'echet function, we need an explicit expression for the log map that is amenable to calculation. This in turn requires a detailed analysis of the geodesics which we carry out in the next section using results from \cite{MO} and \cite{OP}. The resulting expression \eqref{eqn2c} for the log map in Theorem \ref{thm1} and its modification \eqref{eqn3a} are then used in the following two sections which study the limiting distributions for sample Fr\'echet means in ${\boldsymbol T}_{\!m}$; section 3 concentrates on the case when the Fr\'echet means lie in the top-dimensional strata, while section 4 deals with the case when they lie in the strata of co-dimension one. In the final section, we discuss some of the problems involved in generalising our results to the case that the Fr\'echet means lie in strata of arbitrary co-dimension.

\section{The log map on a top-dimensional stratum}

The log map is the generalisation of $\exp^{-1}$, the inverse of the exponential map on a Riemannian manifold. For a tree $T^*$ in ${\boldsymbol T}_{\!m}$ the log map, $\log_{T^*}$, at $T^*$ takes the form
\begin{eqnarray}
\log_{T^*}(T)=d(T^*\!,\,T)\,\v(T)
\label{eqn1}
\end{eqnarray}
as $T$ varies, where $\v(T)$ is a unit vector at $T^*$ along the geodesic from $T^*$ to $T$ and $d(T^*\!,\,T)$ is the distance between $T^*$ and $T$ along that geodesic. This is well-defined since ${\boldsymbol T}_{\!m}$ is a globally non-positively curved space, or $CAT(0)$-space (cf. \cite{BH}), and so this geodesic is unique. 

In order to analyse this log map further, we first recall some relevant aspects of the structure of trees and tree spaces. Apart from the roots and leaves of a tree, which are the vertices of degree 1 mentioned above, there are no vertices of degree two and the remaining vertices, of degree at least 3, are called internal. An edge is called internal if both its vertices are. A tree with $m$ labelled leaves and unspecified internal edge lengths determines a combinatorial type. Then ${\boldsymbol T}_{\!m}$ is a stratified space with a stratum for each such type: a given type with $k\,\,(\leqslant m-2)$ internal edges determines a stratum with $k$ positive parameters ranging over the points of an open $k$-dimensional Euclidean orthant, each point representing the tree with those specific parameters as the lengths of its internal edges. Note that, for this paper, we shall only consider the internal edges of a tree. So by `edge' we always mean `internal edge' and, to simplify the notation, we consider $\Tm$, rather than ${\boldsymbol T}_{\!m}$. 

The metric on $\Tm$ is induced by regarding the identification of a stratum $\tau$ with a Euclidean orthant $\mathcal O$ as an isometry. Then each face, or boundary orthant of co-dimension one, of $\mathcal O$ is identified with a boundary stratum $\sigma$ of $\tau$. A tree of type $\sigma$ is obtained from a tree of type $\tau$ by coalescing the vertices $v_1$ and $v_2$ of degree $p$ and $q$ of the edge whose parameter has become zero, to form a new vertex $v$ of degree $p+q-2$. See Figure \ref{fig1}.
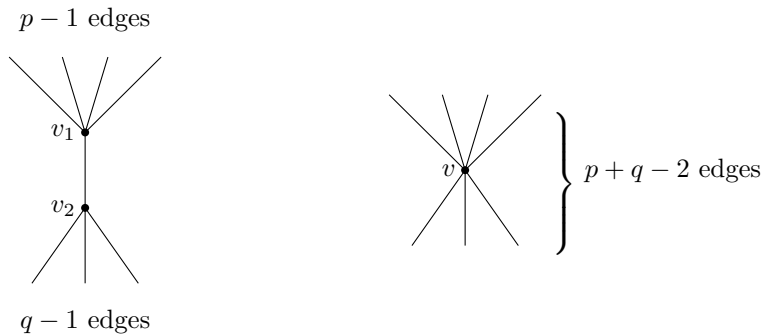
\begin{figure}
\begin{center}
\begin{tikzpicture}
\draw (-5,0.5) -- (-5,-0.5);
\fill (-5,0.5) circle (1.5pt);
\node at (-5,0.5) [left] {$v_1$};
\fill (-5,-0.5) circle (1.5pt);
\node at (-5,-0.5) [left] {$v_2$};
\draw (-6,1.5) -- (-5,0.5);
\draw (-5.3,1.5) -- (-5,0.5);
\draw (-4.7,1.5) -- (-5,0.5);
\draw (-4,1.5) -- (-5,0.5);
\node at (-5,2) {$p-1$ edges};
\draw (-5.7,-1.5) -- (-5,-0.5);
\draw (-5,-1.5) -- (-5,-0.5);
\draw (-4.3,-1.5) -- (-5,-0.5);
\node at (-5,-2) {$q-1$ edges};
\fill (0,0) circle (1.5pt);
\node at (0,0) [left] {$v$};
\draw (-1,1) -- (0,0);
\draw (-0.3,1) -- (0,0);
\draw (0.3,1) -- (0,0);
\draw (1,1) -- (0,0);
\draw (-0.7,-1) -- (0,0);
\draw (-0,-1) -- (0,0);
\draw (0.7,-1) -- (0,0);
\node at (2.5,0) {$\smash{\raisebox{-5pt}{$\left.\rule{0pt}{30pt}\right\}$}}\,\, p+q-2$ edges};
\end{tikzpicture}
\end{center}
\caption{The edge between vertices $v_1$ and $v_2$ shrinks to 0 to form a vertex of degree $p+q+2$.}
\label{fig1}
\end{figure}

We are particularly interested in the top-dimensional strata. These are formed by binary trees, in which all internal vertices have degree 3. A binary tree with $m+2$ leaves has $m+1$ internal vertices and $m$ internal edges so that the corresponding stratum has dimension $m$. There are $(2m+1)!!$ such strata in $\Tm$ (cf. \cite{ES}). For these strata the boundary relation results in two adjacent vertices of degree 3 coalescing to form a vertex of degree 4. Since each vertex of degree 4 can be formed 3 different ways, each stratum of co-dimension one is a component of the boundary of three different top-dimensional strata.  Figure~\ref{fig:3_quadrants} shows an example of these strata in $\T_4$.

\begin{figure}
	\centering
	\includegraphics[width=3in]{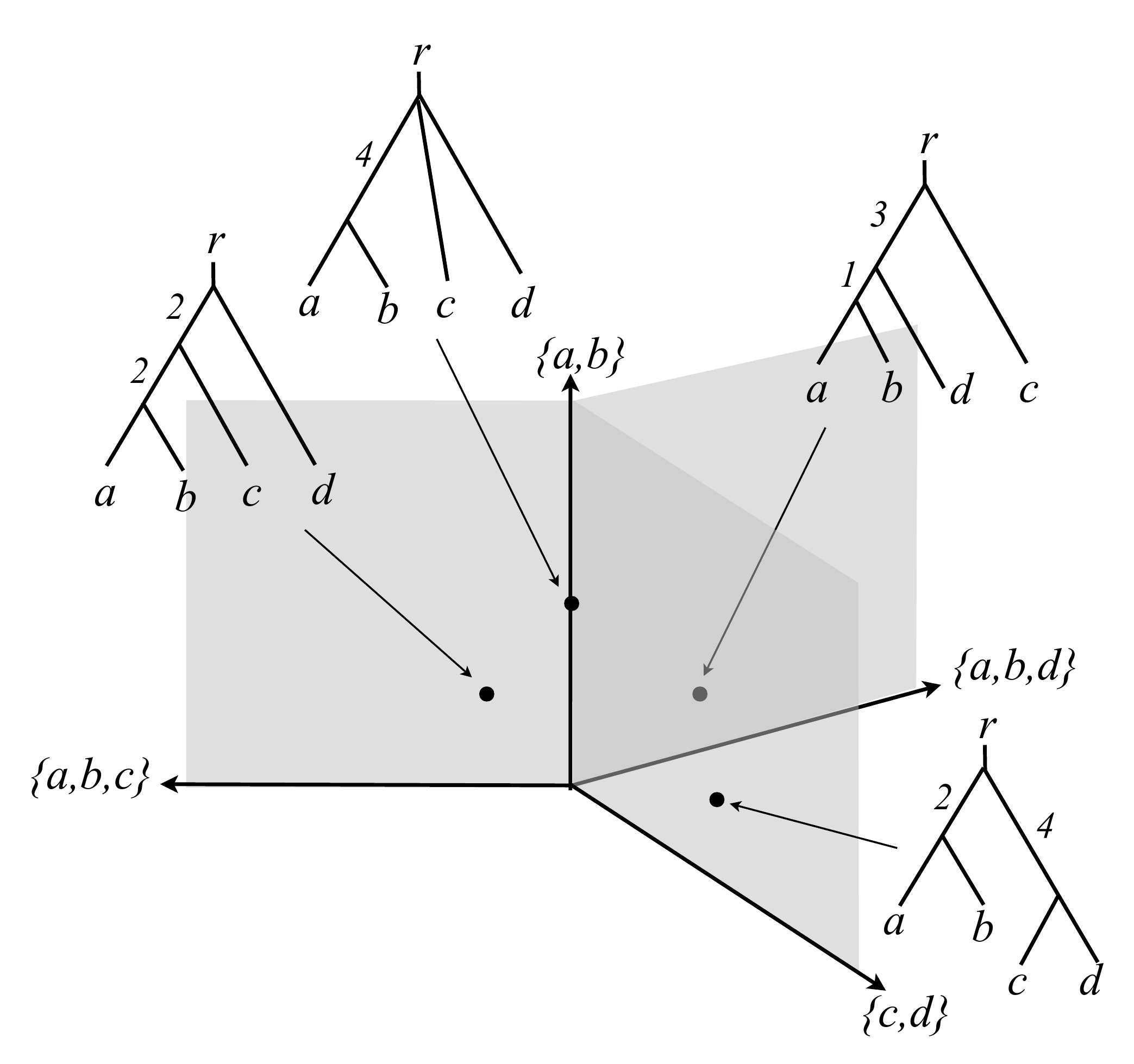}
	\caption{Three adjacent top-dimensional strata in $\T_4$ and their shared co-dimension one stratum.  A sample tree is shown for each stratum, and the axes are labelled by the corresponding edge-type.}
	\label{fig:3_quadrants}
\end{figure}

\vskip 6pt
If a tree $T^*$ lies in a top-dimensional stratum of $\Tm$, since such a stratum can be identified with an orthant $\mathcal O$ in $\mathbb{R}^m$, we may identify the tangent space to $\Tm$ at $T^*$ with $\mathbb{R}^m$. Then, for each point $T\in\Tm$, the geodesic from $T^*$ to $T$ in $\Tm$ will start with a linear segment in $\mathcal O$, which determines an initial \textit{unit} tangent vector $\v(T)\in\mathbb{R}^m$ at $T^*$. Thus, we may identify the image of the log map defined in \eqref{eqn1} as the vector $d(T^*,T)\v(T)$ in $\mathbb{R}^m$. 

\vskip 6pt
For example, the space $\boldsymbol{T}_{\!3}$ of trees with three leaves is the `spider': three half Euclidean lines joined at their origins. Denoting the length of the edge $e$ of $T$ by $|e|_T$, then $d(T^*\!,\,T)=\big||e|_{T^*}-|e|_T\big|$, if $T^*$ and $T$ lie in the same orthant of $\boldsymbol{T}_{\!3}$, and $d(T^*\!,\,T)=|e|_{T^*}+|e|_T$, otherwise. Thus, the log map for $\boldsymbol{T}_{\!3}$ can be expressed explicitly as
\[\log_{T^*}(T)=\left\{\begin{array}{ll}(|e|_T-|e|_{T^*})\e&\hbox{if }T\hbox{ and }T^*\hbox{ are in the same orthant;}\\-(|e|_T+|e|_{T^*})\e&\hbox{otherwise,}\end{array}\right.\]
where $\e$ is the canonical unit vector determining the orthant in which $T^*$ lies. Note that we abuse notation by calling the (single) internal edge $e$ in all 3 trees, despite these edges dividing the leaves in different ways. The explicit expression for the log map for the space $\TT$ of trees with four leaves is already much more complicated than this and was derived in \cite{BLO}. 

\vskip 6pt
To obtain the expression for the log map at $T^*$ for the space $\Tm$ of trees with $m+2$ ($m>2$) leaves, we first summarise without proofs the description, given in \cite{BHV}, \cite{MO}, \cite{OP} and \cite{KV}, of the geodesic between two given trees in $\Tm$.

When an (internal) edge is removed from a tree it splits the set of the leaves plus the root into two disjoint subsets, each having at least two members, and we identify the edges from different trees that induce the same split. Each edge has a `type' that is specified by the subset of the corresponding split that does not contain the root. For example, in the tree in Figure \ref{fig:logmap_Tstar}, the edge labelled $x_3$ has the edge-type $\{a,b\}$, while the edge labelled $x_1$ has the edge-type $\{a,b,c,d\}$. There are 
\begin{eqnarray}
M=2^{m+2}-m-4
\label{eqn1c}
\end{eqnarray}
possible edge-types. Two edge-types are called \textit{compatible} if they can occur in the same tree, and $\Tm$ may be identified with a certain subset of $\mathbb{R}^M$, each possible edge-type being identified with a positive semi-axis in $\mathbb{R}^M$. To make this identification explicit, we choose a canonical order of the edges by first ordering the leaves and then taking the induced lexicographic ordering of the sets of (ordered) leaves that determine the edges. Then, if $\Sigma$ is a set of mutually compatible edge-types and $\O(\Sigma)$ is the orthant spanned by the corresponding semi-axes in $\mathbb{R}^M$, each point of $\O(\Sigma)$ represents a tree with the combinatorial type determined by $\Sigma$ and $\Tm$ is the union of all such orthants. 

For a set of edges $A$ in a tree $T$, define $\|A\|_T= \sqrt{\sum_{e\in A}|e|^2_T}$ and write $|A|$ for the number of edges in $A$. For two given trees $T^*$ and $T$, let $E^*$ and $E$ be their respective edge sets, or sets of non-trivial splits. Assume first that $T^*$ and $T$ have no common edge, i.e. $E^*\cap E=\emptyset$. Then, the geodesic from $T^*$ to $T$ can be determined as follows. 

\begin{lemma}
Let $T^*$ and $T$ be two trees with no common edges, lying in top-dimensional strata of $\Tm$. Then there is an integer $k$, $1 \leqslant k \leqslant m$, and a pair $({\cal A}, {\cal B})$ of partitions ${\cal A} = (A_1,\cdots, A_k)$ of $E^*$ and ${\cal B} = (B_1,\cdots,B_k)$ of $E$, all subsets $A_i$ and $B_j$ being non-empty, such that
\begin{enumerate}
\item[$({\rm P}1)$] for each $i > j$, the union $A_i \cup B_j$ is a set of mutually compatible edges;
\item[$({\rm P}2)$] $\frac{\|A_1\|_{T^*}}{\|B_1\|_T}\leqslant \frac{\|A_2\|_{T^*}}{\|B_2\|_T} \leqslant \cdots \leqslant \frac{\|A_k\|_{T^*}}{\|B_k\|_T}$;
\item[$({\rm P}3)$] for all $(A_i, B_i)$, there are no non-trivial partitions $C_1 \cup C_2$ of $A_i$ and $D_1 \cup D_2$ of $B_i$ such that $C_2 \cup D_1$ is a set of mutually compatible edges and $\frac{\|C_1\|_{T^*}}{\|D_1\|_T} < \frac{\|C_2\|_{T^*}}{\|D_2\|_T}$.
\end{enumerate}
The geodesic is the shortest path through the sequence of orthants $\C=(\O_0,\cdots,\O_k)$ where 
\begin{eqnarray}
\O_i= \O(B_1 \cup \cdots \cup B_i \cup A_{i+1} \cup \cdots \cup A_k)
\label{eqn1a}
\end{eqnarray}
and has length $\| ( \|A_1\|_{T^*} + \|B_1\|_T, \|A_2\|_{T^*} + \|B_2\|_T, \cdots, \|A_k\|_{T^*} + \|B_k\|_T) \|$.
\label{lem0}
\end{lemma}

Note that \eqref{eqn1a} implies that $\O_0=\O(E^*)$ is the orthant in which $T^*$ lies and that $T$ is in $\O_k$. These results, developed from \cite{KV}, are given in this form, though not in a single lemma, in \cite{OP} section 2.3, where the properties (P1), (P2) and (P3) are stated in identical terms. The edge set for $\mathcal{O}_i$ is denoted by $\mathcal{E}^i$ in the statement of Theorem 2.4 there and the formula for the length of the geodesic is equation (1) in that statement.

Following \cite{KV} and \cite{OP} respectively, we call the orthant sequence $\C$ the \emph{carrier} of the geodesic, and the pair of partitions $({\cal A}, {\cal B})$ the \emph{support} of the geodesic. In general, the integer $k$ and the support $({\cal A}, {\cal B})$ need not be unique. However, they are unique if all the inequalities in (P2) are strict \cite[Remark, p.7]{OP} and, in this case, we shall refer to the carrier and support as the \textit{minimal} ones.   

Under the above assumption, the integer $k$ appearing in Lemma \ref{lem0} is the number of times that the geodesic includes a segment in the interior of one orthant followed by a segment in the interior of a neighbouring orthant. Hence, the constraints $1\leqslant k\leqslant m$: $k=1$ implies that the geodesic goes through the cone point and $k=m$ that it passes through a sequence of top-dimensional orthants.

We can now give an isometric embedding $\tilde{\C}$ in $\mathbb{R}^m$ of $\C\subseteq\mathbb{R}^M$ with $T^*$ mapped to $\u^*=(u^*_1,\cdots,u^*_m)$ in the positive orthant, where the $u^*_i>0$ represent the lengths of the edges of $T^*$, and with $T$ mapped to $\u=-(u_1,\cdots,u_m)$ in the negative orthant, where the $u_i>0$ are the lengths of the edges of $T$.  Let $(t^*_1,\cdots,t^*_m)$ be the coordinates of $T^*$ ordered by the canonical ordering given just before Lemma~\ref{lem0} that embeds $\Tm$ in $\mathbb{R}^M$.  Then we can reorder the coordinates $u^*_i$ such that the edges in $A_1$ correspond to the first $| A_1 |$ positive semi-axes in $\mathbb{R}^m$, the edges in $A_2$ correspond to the next $|A_2|$ positive semi-axes in $\mathbb{R}^m$, etc, while the edges in $B_1$ correspond to the first $|B_1|$ negative semi-axes in $\mathbb{R}^m$, the edges in $B_2$ correspond to the next $|B_2|$ negative semi-axes in $\mathbb{R}^m$, etc.  By (P1), the edge sets $B_1,\cdots, B_i, A_{i+1},\cdots, A_k$ are mutually compatible for all $0 \leqslant i \leqslant k$, implying that the images of these edges in $\mathbb{R}^m$ are mutually orthogonal, and so they determine an isometric embedding of $\O_i$, defined by \eqref{eqn1a}, and hence the required isometric embedding $\tilde\C$ of $\C$.  Let $\pi$ be the inverse of the permutation of the coordinates described above, so that
\begin{eqnarray}
\pi:\u^*=(u^*_1,\cdots,u^*_m)\mapsto\t^*=(t^*_1,\cdots,t^*_m).
\label{eqn1b}
\end{eqnarray}  

\begin{ex} \label{ex:logmap}
Figure~\ref{fig:logmap_1} shows the embedded geodesic and minimal carrier between the trees $T^*$ and $T$ (see Figures~\ref{fig:logmap_Tstar} and~\ref{fig:logmap_T}), which correspond to the points $\u^*$ and $\u$, respectively.  The minimal support consists of $A_1 = \{u_1^*,u_2^*\}$, $A_2 = \{u_3^*\}$, $B_1 = \{-u_2\}$, and $B_2 = \{-u_1, -u_3\}$.  For convenience, $\pi$ is the identity permutation in this case.  The minimal carrier consists of the all positive octant determined by $x_1>0$, $x_2>0$ and $x_3>0$; the 2-dimensional quadrant formed by the positive $x_3$ and negative $x_2$ axes; and the all negative octant.
\end{ex}

\begin{figure}
	\begin{subfigure}{.5\textwidth}
		\centering
		\includegraphics[width=1in]{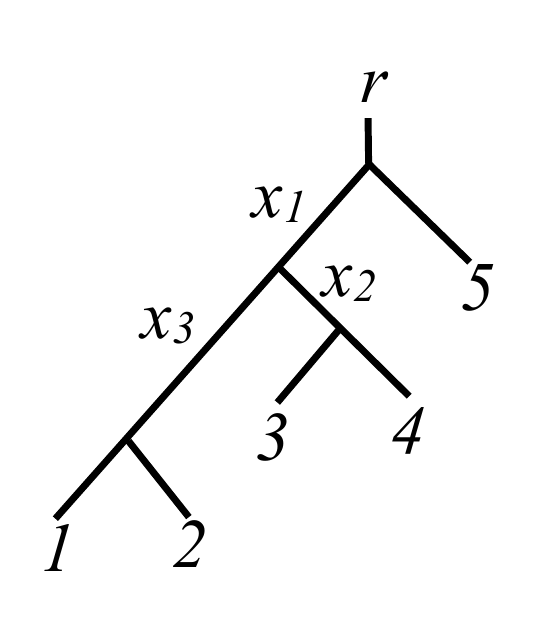}
		\caption{Tree $T^*$}
		\label{fig:logmap_Tstar}
	\end{subfigure}
	\begin{subfigure}{.5\textwidth}
		\centering
		\includegraphics[width=1in]{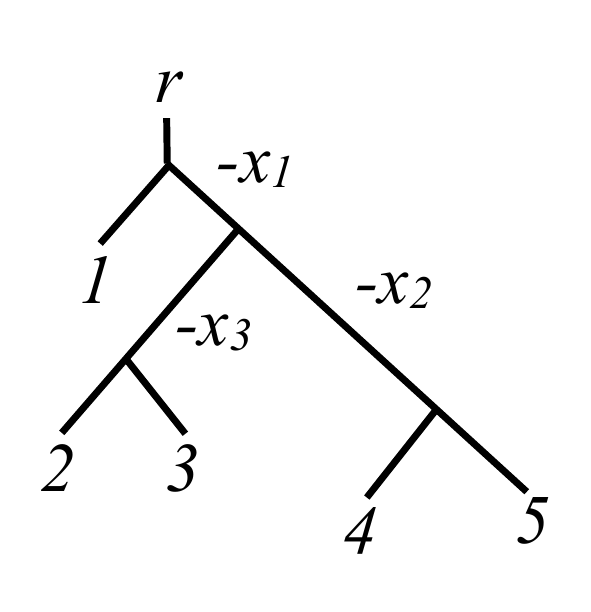}
		\caption{Tree $T$}
		\label{fig:logmap_T}
	\end{subfigure}
	\begin{subfigure}{1\textwidth}
		\centering
		\includegraphics[width=3in]{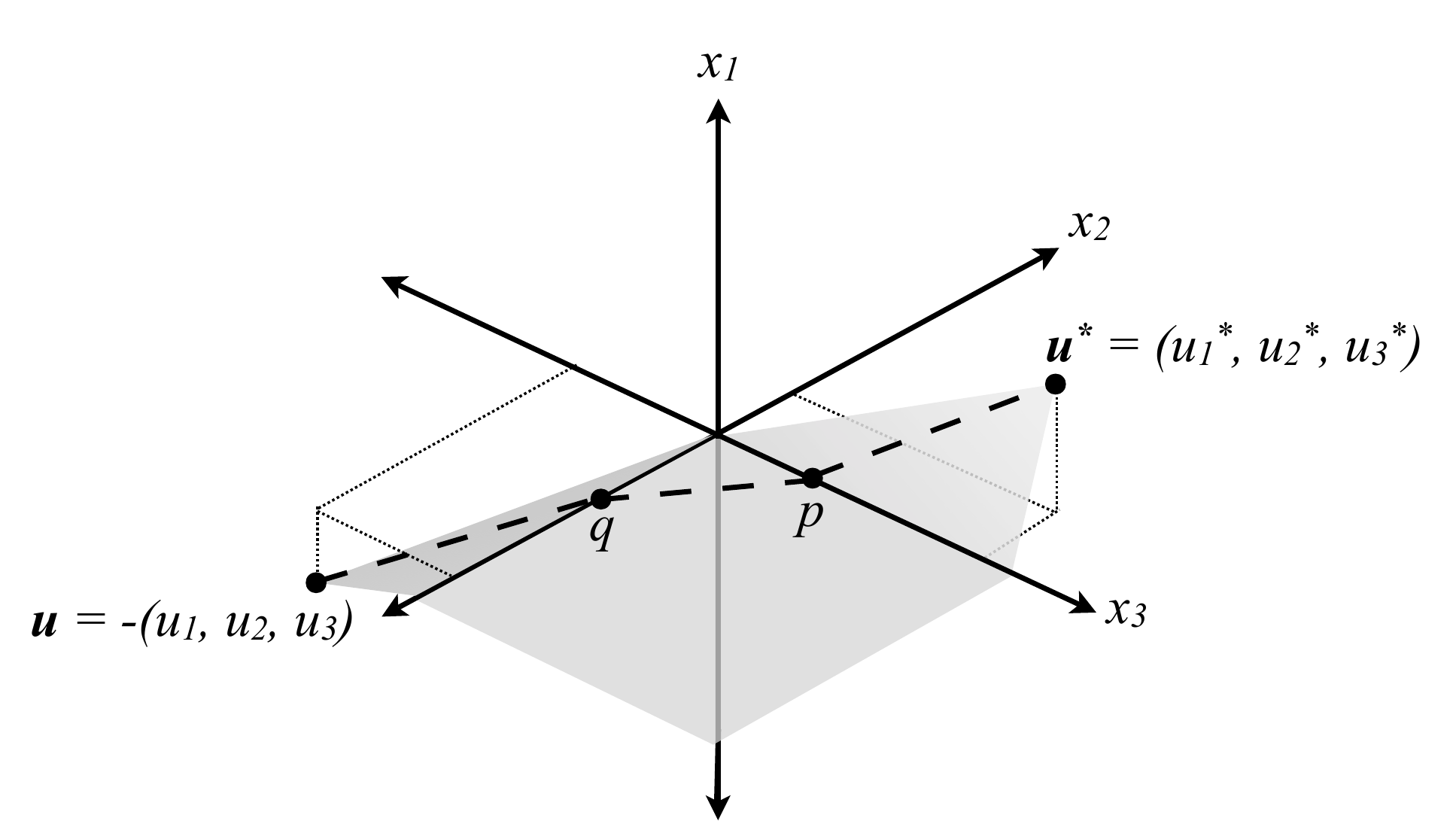}
		\caption{The geodesic between the trees corresponding to $\u^*$ and $\u$ is marked with the dashed line.  The $-x_1$,$-x_2$,$x_3$ octant does not exist in tree space, but the $-x_2$, $x_3$ quadrant does, so the geodesic is restricted to lying in the grey area. It bends at the points $p$ and $q$.}
		\label{fig:logmap_1}
	\end{subfigure}
	\begin{subfigure}{1\textwidth}
		\centering
		\includegraphics[width=3in]{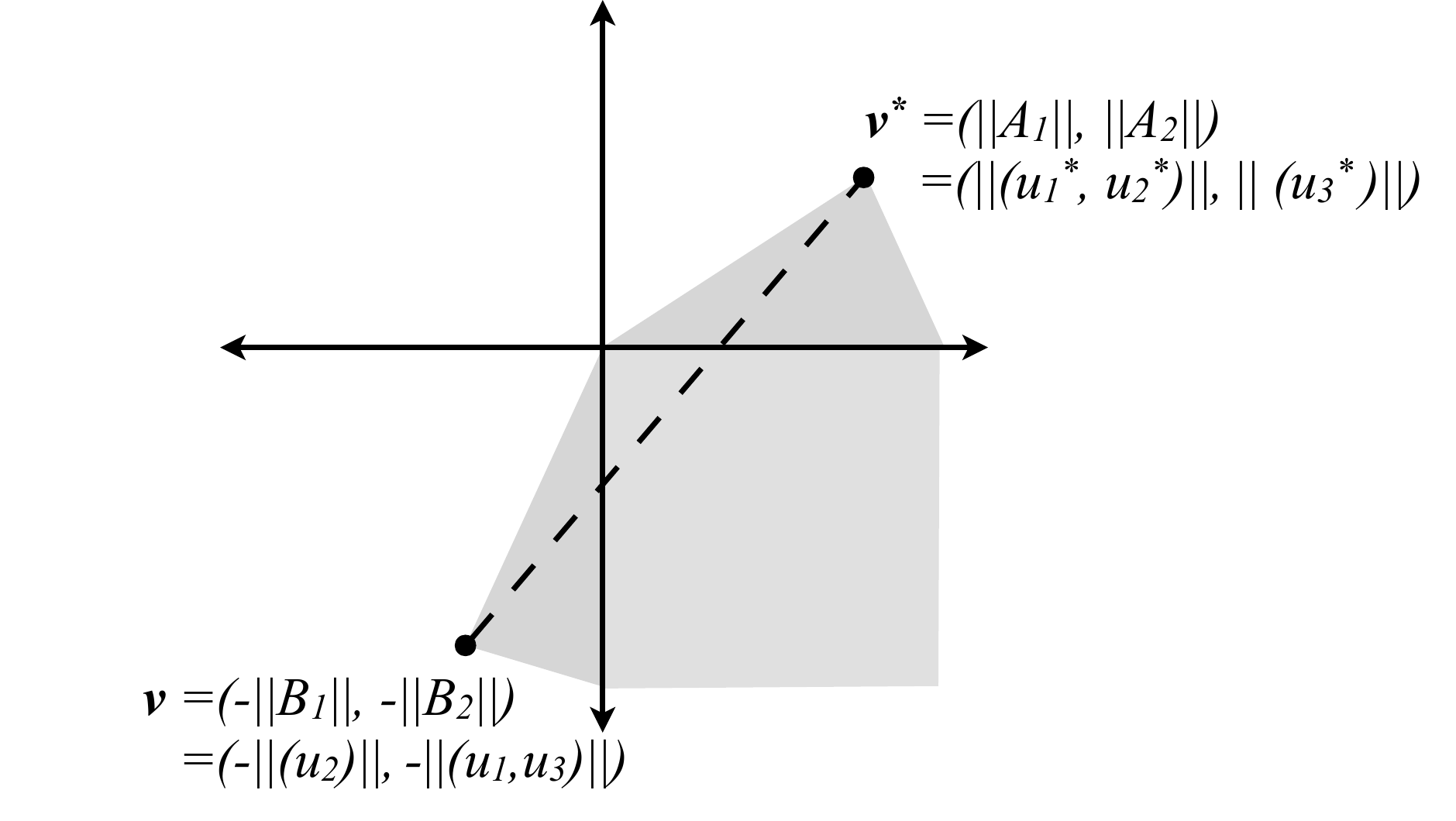}
		\caption{The isometric embedding of the grey area in (c) into $V^2$.  Intuitively, this corresponds to ``unfolding" the bends.  $\u^*$ and $\u$ are mapped to $\v^*$ and $\v$.  The Euclidean geodesic between $\v^*$ and $\v$ in $V^2$ is contained in the grey area, and thus can be mapped back onto the geodesic in tree space.}
		\label{fig:logmap_2}
	\end{subfigure}
	\caption{Trees, carrier, and isometric embedding for Example 1.}
	\label{fig:logmap_ex1}
\end{figure}

For any $1\leqslant l\leqslant m$, let $V^l$ be the subspace of $\mathbb{R}^l$ that is the union of the (closed) orthants $\mathcal{P}_i$, $i=0,\cdots,l$, where 
\[\mathcal{P}_i=\{(x_1,\cdots,x_l)\in\mathbb{R}^l\mid x_j\leqslant0\hbox{ for }j\leqslant i\hbox{ and }x_j\geqslant0\hbox{ for }j>i\}.\] 
For the given $T^*$, $T$, and corresponding $k$ from Lemma \ref{lem0}, there are $k+1$ orthants in the carrier of the geodesic between $T^*$ and $T$. If $k = m$ (the intrinsic dimension of $\Tm$), then the carrier $\C$ is isometric to $\tilde{\C}=V^m$, with $\O_i$ coinciding with $\mathcal{P}_i$ and with the geodesic from $\u^*$ to $\u$ being a straight line contained in $\tilde{\C}$. Otherwise if $k < m$, the space $\tilde{\C}$ is strictly contained in $V^m$, and some of the top-dimensional orthants of $V^m$ may not correspond to orthants in tree space. Additionally, the geodesic between $\u^*$ and $\u$ in $\tilde{\C}$ will bend at certain orthant boundaries within the ambient space $V^m$. We now give an isometric embedding onto $V^k$ of a subspace of $\tilde{\mathcal C}$ containing the geodesic in $V^m$ such that the image geodesic is a straight line.

The geodesic between $\u^*$ and $\u$ passes through $k$ orthant boundaries. At the $i$-th orthant boundary, the edges in $A_i$, which have been shrinking in length since the geodesic started at $\u^*$, simultaneously reach length 0, and the edges in $B_i$ simultaneously appear in the tree with length 0 and start to grow in length. The length of each edge in $A_i$ changes linearly as we move along the geodesic, and thus since these lengths all reach 0 at the same point, the ratios of these lengths to each other remain the same along the geodesic.  An analogous statement can be made for the lengths of the edges in $B_i$ (cf. \cite{MO} Corollary 4.3).  The basic idea behind the embedding into $V^k$ is that because the lengths of the edges in $A_i$, for any $i$, are all linearly dependent on each other, we can represent those edges in $V^k$ using only one dimension, and analogously for the edges in $B_i$.  

More specifically, for $1 \leqslant i \leqslant k$, let 
\[\v_i^* = (u_{|A_1| + \cdots + |A_{i-1}| + 1}^*, \cdots, u_{|A_1| + \cdots + |A_{i-1}| + |A_i |}^*),\]
be the projection of $\u^*$ on the orthant $\O(A_i)$. That is, the coordinates of $\v_i^*$ are the lengths of the edges in $A_i$, ordered as chosen above. Similarly, let 
\[\v_i = (u_{|B_1| + \cdots + |B_{i-1}| + 1}, \cdots, u_{|B_1| + \cdots + |B_{i-1}| + |B_i |}),\]
so that the coordinates of $\v_i$ are the lengths of the edges in $B_i$, in that order. Then, the geodesic between $T^*$ and $T$ in $\C$ is piece-wise linearly isometric with the Euclidean geodesic between the vectors 
\[\v^*=(\|\v_1^*\|,\cdots,\|\v_k^*\|) = (\|A_1\|_{T^*},\cdots,\|A_k\|_{T^*}) \]
and
\[\v=(-\|\v_1\|,\cdots,-\|\v_k\|) = (-\|B_1\|_T,\cdots,-\|B_k\|_T)\]
in $V^k$, and hence in $\mathbb{R}^k$. In particular, the Euclidean distance between these two Euclidean points is the same as the distance between $T^*$ and $T$ in $\C$. Thus we have the following result, the essence of which appears in \cite{MO}, Theorem 4.10, to which we refer readers for more detailed proof.

\begin{lemma}
For any given $T^*$ and $T$ in $\Tm$ with no common edge and with $T^*$ lying in a top-dimensional stratum, there is an integer $k$, $1\leqslant k\leqslant m$, for which there are two vectors $\v^*,\v\in\mathbb{R}^k$, depending on both $T^*$ and $T$, such that the geodesic between $T^*$ and $T$ is homeomorphic and piece-wise linearly isometric, with the $($straight$)$ Euclidean geodesic between $\v^*$ and $\v$, where $\v^*$ lies in the positive orthant of $\mathbb{R}^k$ and $\v$ in the closure of the negative orthant.
\label{lem1}
\end{lemma}

For Example~\ref{ex:logmap}, $k = 2$, and thus the grey area shown in Figure~\ref{fig:logmap_1} is isometrically mapped to $V^2$, as shown in Figure~\ref{fig:logmap_2}.

\vskip 6pt
In the general case where $T^*$ and $T$ have a common edge, say $e$, this common edge determines, for each of the two trees, two quotient trees $T^*_i$ and $T_i$, $i=1,2$, described as follows (cf. \cite{MO} \& \cite{KV}). The trees $T_1^*$ and $T_1$ are obtained by replacing the subtree `below' $e$ with a single new leaf, so that $e$ becomes an external edge. These two replaced subtrees form the trees $T_2^*$ and $T_2$, with the `upper' vertex of the edge $e$ becoming the new root. Then, the geodesic $\gamma(t)$ between $T^*$ and $T$ is isometric with $(\gamma_e(t),\gamma_1(t),\gamma_2(t))$, where $\gamma_e$ is the linear path from $|e|_{T^*}$ to $|e|_T$ and $\gamma_i$ is the geodesic from $T^*_i$ to $T_i$ in the corresponding tree space. For this, we treat $\T_{\!1}$ and $\T_{\!2}$, the spaces of trees with no internal edges, as single points, so that any geodesic in them is a constant path. Assuming that $T^*_i$ and $T_i$ have no common edge for $i=1$ or 2, we may obtain, as above, a straightened image of each geodesic $\gamma_i$ in $V^{k_i}$ with $T^*_i$ represented in the positive orthant and $T_i$ in the negative one. Combining these with the geodesic $\gamma_e$, which is already a straight linear segment, we have an isometric representation of $\gamma$ as a straight linear segment in $\mathbb{R}_+\times V^{k_1} \times V^{k_2}$. In this case, the sequence of strata containing the tree space geodesic between $T^*$ and $T$ is contained in the product of the carriers for the relevant quotient trees, together with an additional factor for the common edge. For example, if $0< t_1 < t_2 < t_3 < t_4 < 1$ and the geodesic $\gamma_1$ spends $[0,t_2]$ in orthant $\O_1$, $[t_2,t_3]$ in orthant $\O_2$, $[t_3,1]$ in orthant $\O_3$, while the geodesic $\gamma_2$ spends $[0,t_1]$ in orthant $\P_1$, $[t_1,t_4]$ in orthant $\P_2$, $[t_4,1]$ in orthant $\P_3$, then the carrier for the product geodesic would be the sub-sequence 
\[\O_1\times\P_1,\O_1\times\P_2,\O_2\times\P_2,\O_3\times\P_2,  \O_3\times\P_3\]
of the full lexicographically ordered sequence of nine products. 

If $T^*$ and $T$ have more than one common edge, then either $T^*_1$ and $T_1$, or $T^*_2$ and $T_2$, will have a common edge and we may repeat the process. Having done so as often as necessary, we arrive at a sequence of orthants determined by the non-common edges of $T^*$ and $T$. These we relabel $\O_0$ to $\O_k$ as in Lemma \ref{lem0}. If $\O_{-1}$ is the orthant determined by the axes corresponding to the common edges of $T^*$ and $T$, then the sequence
\[\O_{-1}\times\O_0,\,\O_{-1}\times\O_1,\,\cdots,\,\O_{-1}\times\O_k\]
is the carrier of the geodesic from $T^*$ to $T$. Similarly, the support for the tree space geodesic between $T^*$ and $T$ is found by interleaving the partitions in the supports of the relevant quotient trees so that property (P2) is satisfied in the combined support. The resulting partitions $\mathcal{A}$ and $\mathcal{B}$ are then preceded by the set $A_0=B_0$ of axes corresponding to the common edges so that $\O_{-1}=\O(A_0)=\O(B_0)$, with the convention that the corresponding ratio is $-\|A_0\|_{T^*}/\|B_0\|_T$, and \eqref{eqn1a} is modified to
\[\O_i=\O(B_0\cup B_1\cup\cdots\cup B_i\cup A_{i+1}\cup\cdots\cup A_k).\]
In this generalised context, the value $k=0$ is now possible, implying that all edges are common to $T^*$ and $T$. In other words, they lie in the same orthant. Note that this presentation differs slightly from that in Section 1.2 of \cite{MOP} in that, by collecting all the common edges in a single member $A_0=B_0$ of the support, we are implicitly suppressing the axiom (P3) for that set. Note that the maximum value of the number $k$, which is determined by the non-common edges of $T^*$ and $T$, is $m-|A_0|$ in the general case.

\begin{definition}
We call $k$, the number of changes of orthant in the unique minimal carrier of the geodesic from $T^*$ to $T$, the carrier number $k(T^*,T)$ of $T^*$ and $T$. 
\label{def0}
\end{definition}

Clearly, $k(T^*,T)=k(T,T^*)$.

\vskip 6pt
The minimal carrier and support determine the corresponding $\u^*$, $\v^*_i$, $\v^*$ and $\v$ in a similar manner to the special case where there is no common edge between $T^*$ and $T$ given in Lemma \ref{lem1}, modified to account for the common edges. For this, the first $|A_0|$ coordinates of $\u^*$ will be the $(A_0)_{T^*}$ and those of $\u$ will be $+(B_0)_T$; for $k=k(T^*,T)$ and $1\leqslant i\leqslant k$,
\begin{eqnarray}
\v^*_i=\left(u_{|A_0|+|A_1| + \cdots + |A_{i-1}| + 1}^*, \cdots, u_{|A_0|+|A_1| + \cdots + |A_{i-1}| + |A_i |}^*\right)
\label{eqn2b}
\end{eqnarray}
and $\v_i$ is modified similarly; and $\v^*$ and $\v$ have additional first coordinates $\v_0^*=(A_0)_{T^*}$ and $\v_0=(B_0)_T$ respectively. Then, with this modification, the geodesic between $T^*$ and $T$ is homeomorphic and piece-wise linearly isometric, with the (straight) Euclidean geodesic between $\v^*$ and $\v$, where $\v^*$ lies in the positive orthant of $\mathbb{R}^{|A_0|}_+\times\mathbb{R}^k$. This generalisation of Lemma \ref{lem1} to the general case was obtained, with different notation, in \cite{BHV}, \cite{MO} and \cite{KV}. Then, the log map as defined by \eqref{eqn1} can be expressed using these vectors as follows.
  
\begin{theorem}
Fix $T^*$ in a top-dimensional stratum of $\Tm$ with coordinates $\t^*=(t^*_1,\cdots,t^*_m)$, where the ordering of the coordinates is that induced by the canonical ordering for $\mathbb{R}^M$. 
For $T\in\Tm$, there are vectors $\v^*$ and $\v$ in $\mathbb{R}^{|A_0|}_+\times\mathbb{R}^k$, where $|A_0|$ is the number of common edges of $T^*$ and $T$, $k$ is the carrier number $k(T^*,T)$ and $\v^*$ lies in the positive orthant of the corresponding space, and a linear map $\rho$ such that 
\begin{eqnarray}
\log_{T^*}(T)=\rho(\v-\v^*)=\rho(\v)-\t^*.
\label{eqn2c}
\end{eqnarray}
\label{thm1}
\end{theorem}

\begin{proof}
Let $\v^*$ have $(|A_0|+i)$th coordinate $\|\v^*_i\|$, $i=1,\cdots,k$, with the additional initial coordinates $\v^*_0$ when $T^*$ and $T$ have common edges, where $\v_i^*$ are as defined by \eqref{eqn2b}, and $\v$ be determined similarly. The piece-wise linear isometry that straightens the geodesic from $T^*$ to $T$, given in Lemma \ref{lem1} for the special case as well as the above for the general case, has an inverse on the positive orthant in $\mathbb{R}^{|A_0|}_+\times\mathbb{R}^k$. This inverse is given by
\begin{eqnarray}
\chi:\,\e_i\mapsto\frac{1}{\|\v^*_i\|}\v^*_i\qquad1\leqslant i\leqslant k,
\label{eqn2}
\end{eqnarray}
and the identity on the $|A_0|$ initial coordinates, where $\e_i$ is the $(|A_0|+i)$th standard basis vector in $\mathbb{R}^{|A_0|}_+\times\mathbb{R}^k$. Note that, being a linear map, when $A_0=\emptyset$, $\chi((x_1,\cdots,x_k))=\displaystyle\sum_{i=1}^kx_i\frac{1}{\|\v^*_i\|}\v^*_i$, where $(x_1,\cdots,x_k)=\displaystyle\sum_{i=1}^kx_i\e_i\in\mathbb{R}^k$. Although it is not expressed precisely as it is here, the idea for a more detailed derivation of this in this case is captured in Theorem~4.4 in \cite{MO}, where $\chi$ is denoted by $g_0$.

Since $\v^*=\v^*_0+\sum\limits_{i=1}^k\|\v^*_i\|\,\e_i$, we have that $\chi(\v^*)=\u^*$ and that $\chi$ maps the initial segment of the straight geodesic in $\mathbb{R}^{|A_0|}_+\times\mathbb{R}^k$, together with its initial tangent vector $\v-\v^*$, onto those of the geodesic in $V^m$. The permutation $\pi$, which maps the positive orthant in $V^m$ into $\T_m\subset\mathbb{R}^M$ where $M$ is defined by \eqref{eqn1c}, is also an isometry preserving the initial segments of the geodesics. It follows that 
\begin{eqnarray}
\log_{T^*}(T)=\pi\circ\chi(\v-\v^*).
\label{eqn2a}
\end{eqnarray}
Noting that the maps $\pi$ and $\chi$ are linear and $\pi\circ\chi(\v^*)=t^*$, the required result follows by taking $\rho=\pi\circ\chi$.  
\end{proof}

Figure~\ref{fig:logmap_3} shows the log map for the tree $T^*$ for Example~\ref{ex:logmap}.
\begin{figure}
		\centering
		\includegraphics[width=3in]{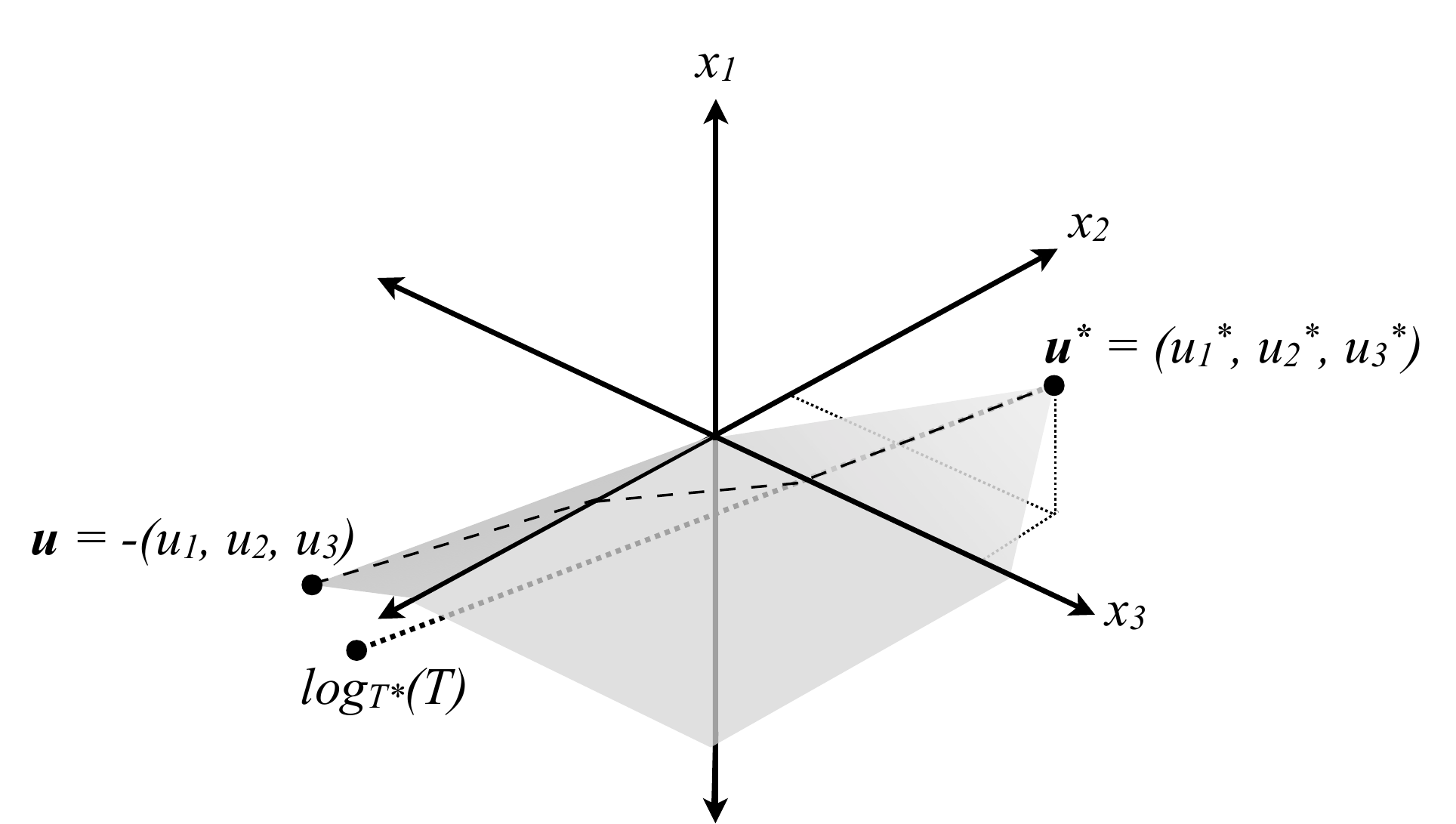}
		\caption{The log map for tree $T^*$ in Example~\ref{ex:logmap}.
The vector between $\u^*$ and $log_{T^*}(T)$ is shown as a dashed line. It coincides with the geodesic between $T^*$ and $T$ in the starting orthant, but then continues into the ambient space, while the geodesic must bend to remain in the tree space.}
		\label{fig:logmap_3}
\end{figure}

\vskip 6pt
Although $\Tm$ is $CAT(0)$, $\log_{T^*}$ is not a one-to-one map. In particular, if $T_1$ and $T_2$ are two different trees such that $k(T^*,T_1)=k(T^*,T_2)$ is not maximal, then it is possible that $\log_{T^*}(T_1)=\log_{T^*}(T_2)$, as observed in the case of $\TT$ in \cite{BLO}. As another example, consider two trees $T^*, T \in \Tm$ with no common edges such that the geodesic between them passes through the cone point with a given length $l$.  Then for any other tree $T'$ with a geodesic to $T^*$ of length $l$, passing through the cone point, we also have that $\log_{T^*}(T)=\log_{T^*}(T')$.

Recalling that each component of $\v^*_i$ and $\v_i$ is respectively the length of an edge in $A_i$ and $B_i$ then, with some ambiguity in the ordering of the edges of $T^*$, another equivalent way to express $\log_{T^*}$ is
\[\log_{T^*}(T)=\{\bar B_0-\bar A_0\}-\sum_{j=1}^k\frac{\|B_j\|_T+\|A_j\|_{T^*}}{\|A_j\|_{T^*}}\bar A_j\]
where $\bar A_j=(e_{T^*})_{e\in A_j}$. To derive the limiting distribution of sample Fr\'echet means, the ordering must be kept explicit and independent of $T$.  Hence, we have to use the expression for the log map given by \eqref{eqn2a}, even though it is not as transparent as this one.

\vskip 6pt
Note also that, although the definitions for both $\pi$ and  $\chi$ implicitly depend on the ordering we chose for the coordinates of $\u^*$, the composition $\pi\circ\chi$ is independent of that choice, and so the log map is well-defined, as long as we chose the same ordering for $\u^*$ for both $\pi$ and $\chi$. 

\vskip 6pt
The minimal carrier that determines the maps $\pi$ and $\chi$ as well as the vectors $\v^*$ and $\v$ depends on both $T^*$ and $T$, although we have suppressed that dependence in the notation. However, there are only finitely many choices for the carrier number $k(T^*,T)$ and the minimal support when $T^*$ is fixed and $T$ varies within a given stratum of $\Tm$. In particular, if $k(T^*,T)$ remains constant in a neighbourhood of $(T^*,T)$, then $\pi$ and $\chi$ do not change for small enough changes in $T^*$ and $T$. It follows that there are only finitely many possibilities for the form \eqref{eqn2c} that $\pi\circ\chi$ takes when $T$ varies in $\Tm$. Here, by form, we mean the algebraic expression of $\log_{T^*}$ as a map. That is, by `$\log_{T^*}(T_1)$ and $\log_{T^*}(T_2)$ taking the same form', we mean that they can be obtained using a single algebraic expression for $\log_{T^*}$. Since the permutation $\pi$ returns all the axes to their canonical order, this expression is determined by the partition $\mathcal{A}$ of the edges of $T^*$, with the subsets of non-common edges possibly permuted. For example, in the case of $\TT$, $\log_{T^*}$ only takes two possible forms, depending on whether the geodesic from $T^*$ to $T$ passes through the cone point or not where the cone point, the origin in $\mathbb{R}^M$, represents the tree whose two edges have zero length. The two corresponding subsets of $\TT$ are respectively indicated by the unions of light and dark grey regions in Figure 3 of \cite{BLO} when $T^*$ is the tree corresponding to $(x_i,x_j)$. The different possibilities for the form \eqref{eqn2c} give rise to a \textit{polyhedral subdivision} of tree space $\Tm$, defined as follows.

\begin{definition}
For a fixed $T^*$ lying in a top-dimensional stratum of $\Tm$, the polyhedral subdivision of tree space $\Tm$, with respect to $T^*$, is determined by the possible forms that $\log_{T^*}$ can take: each polyhedron of the subdivision is the closure of the set of trees $T$ that have a particular form for $\log_{T^*}(T)$. We shall call each such top-dimensional polyhedron a maximal cell of the polyhedral subdivision and let $\mathcal{D}_{T^*}$ be the subset of $\Tm$ consisting of all trees that lie on the boundaries of maximal cells determined by the polyhedral subdivision with respect to $T^*$.
\label{def1}
\end{definition}

Note that, if the geodesics to $T_1$ and $T_2$ from $T^*$ pass through the same sequence of strata, then $\log_{T^*}(T_1)$ and $\log_{T^*}(T_2)$ take the same form. However, the converse is not always true. For example, it is possible that $T_1$ and $T_2$ lie in different strata, but in the same maximal cell. Hence, the definition of polyhedral subdivision of $\Tm$ defined here is similar to, but coarser than, the concept of `vistal polyhedral subdivision' given in section 3 of \cite{MOP}. This is due to the fact that, while $\mathcal A$ and $\mathcal B$ in the minimal support play a symmetric role for the geodesic between $T^*$ and $T$, their roles in the log map $\log_{T^*}$ are asymmetric. When $T$ varies, as long as the corresponding partition $\mathcal A$ either is unchanged or, at most, its subsets corresponding to the non-common edges are permuted, the algebraic expression for $\log_{T^*}$ remains the same.

This polyhedral subdivision varies continuously with respect to $T^*$. If $T$ lies in the interior of a maximal cell of the subdivision and $T^*$, itself in a top-dimensional (open) stratum, varies in a small enough neighbourhood, then the support for $T^*$ and $T$ is unique. Then, the derivative of the log map will be well-defined.

When $T$ lies on the boundary of a maximal cell of the subdivision, but not on a stratum boundary, the possible supports for $T^*$ and $T$ are those determined by the polyhedra to which that boundary belongs. However, all of these supports give rise to the same geodesic between $T^*$ and $T$, as they must since $\Tm$ is a $CAT(0)$-space, and among them will be the minimal support that we are assuming for our analysis. Moreover, in this case, there is at least one non-minimal support for $T^*$ and $T$ with the property that, for the corresponding $\v^*$ and $\v$, $\|\v^*_i\|/\|\v_i\|=\|\v^*_{i+1}\|/\|\v_{i+1}\|$ for some $i\geqslant1$. 

\vskip 6pt
Recall that from Definition \ref{def0} that the carrier number counts the number of orthants that the geodesic from $T^*$ to $T$ meets in a linear segment of positive length. It will become clear later that the set of trees $T$ for which, for a given $T^*$, the carrier number $k(T^*,T)$ is less than its possible maximum $m-|A_0|$, where $|A_0|$ is the number of common edges of $T^*$ and $T$, plays a role that distinguishes the limiting distributions of sample Fr\'echet means in the tree spaces from those in Euclidean space. Hence, we introduce the following definition.

\begin{definition}
A point $T\in\Tm$ is called singular, with respect to a tree $T^*$ lying in a top-dimensional stratum, if the carrier number $k(T^*,T)$ of $T^*$ and $T$ is less than $m-|A_0|$. The set of such singular points will be denoted by $\mathcal{S}^{\phantom{A}}_{T^*}$. 
\label{def2}
\end{definition}

The following result describes the image, under $\log_{T^*}$, in the tangent space at $T^*$ of the set $\mathcal{S}^{\phantom{A}}_{T^*}$: although $\mathcal{S}^{\phantom{A}}_{T^*}$ may be rather complex, its image is relatively simple.

\begin{corollary}
If $T^*\in\Tm$ lies in a top-dimensional stratum, then the image, under $\log_{T^*}$, of the set $\mathcal{S}^{\phantom{A}}_{T^*}$ of the singular points with respect to $T^*$ is contained in the union of the hyperplanes $x_it^*_j=x_jt^*_i$, $1\leqslant i\not=j\leqslant m$, in $\mathbb{R}^m$.
\label{cor1}
\end{corollary}

\begin{proof}
The number of orthants in the minimal carrier of the geodesic from $T^*$ to $T$ is less than $m-|A_0|$ if and only if the dimension $j_i$ of some vector $\v^*_i$ is greater than one for $i\geqslant1$. Then, $\chi$ maps the line determined by $\e_i$ in $\mathbb{R}_+^{|A_0|}\times\mathbb{R}^k$ into the subspace of $\mathbb{R}^m$ that is the intersection of the co-dimension one hyperplanes $x_{i'}u^*_{j'}=x_{j'}u^*_{i'}$ in $\mathbb{R}^m$, where $j_1+\cdots+j_{i-1}< i'\not=j'\leqslant j_1+\cdots+j_i$ and where the ordering of the coordinates $u^*_{i'}$, and hence of the $x_{i'}$, is as in the minimal carrier. Then, applying the permutation $\pi$ and using the same notation for the permuted $\boldsymbol x$-coordinates, the result follows.
\end{proof}

\begin{figure}
		\centering
		\includegraphics[width=3in]{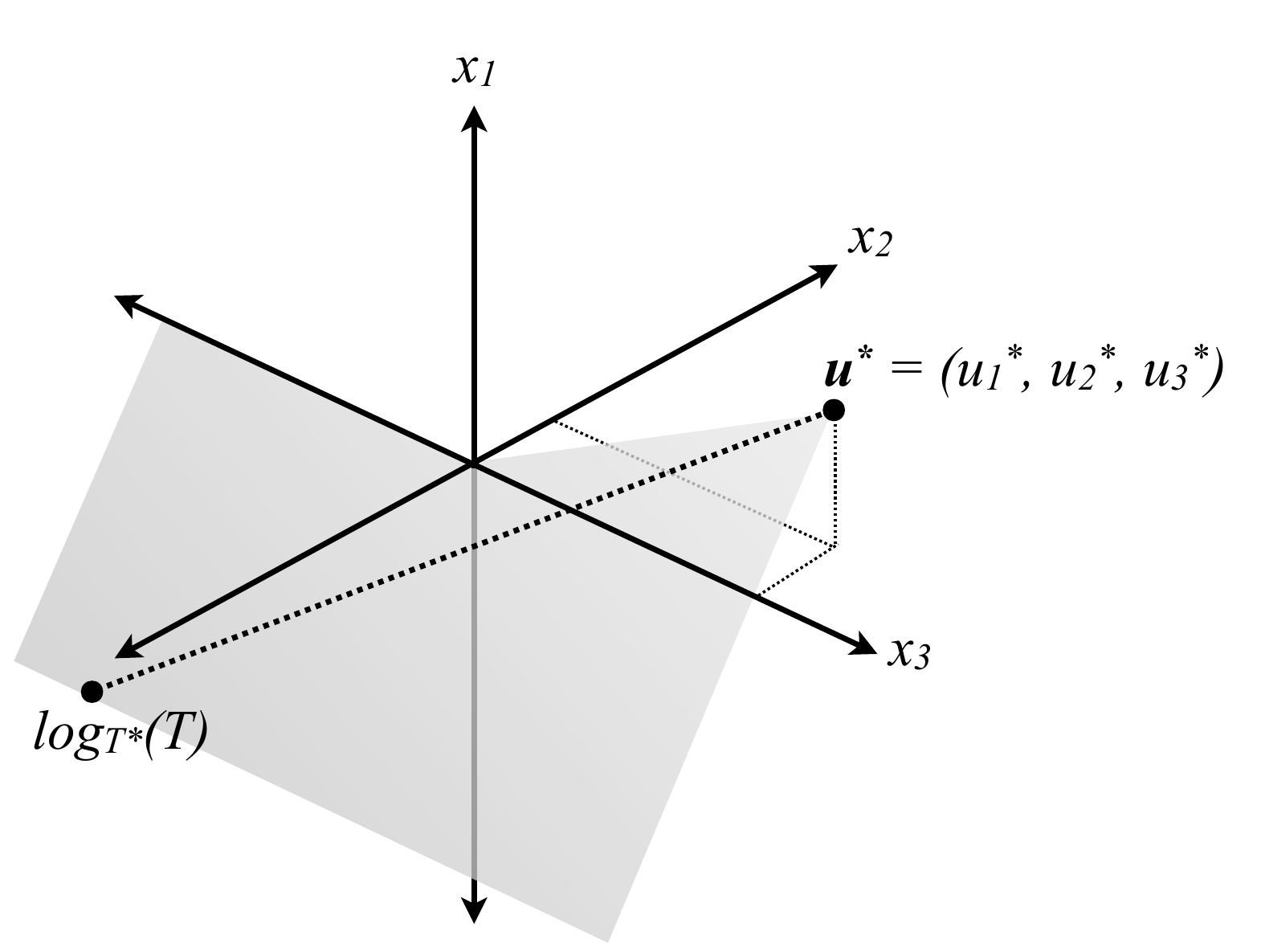}
		\caption{The grey area is part of the hyperplane $x_1 \cdot u^*_2 =  x_2 \cdot u^*_1$, which contains some of the singular points for the log map $log_{T^*}$ for Example~\ref{ex:logmap}.}
		\label{fig:logmap_4}
\end{figure}
For example, see Figure~\ref{fig:logmap_4} for an illustration of one of the hyperplanes for Example~\ref{ex:logmap}.

\vskip 6pt
To describe the limiting behaviour of sample Fr\'echet means, it will be more convenient to have a modified version of the log map, $\Phi_{T^*}$, at $T^*$ defined by 
\begin{eqnarray}
\Phi_{T^*}(T)=\log_{T^*}(T)+\t^*.
\label{eqn3a}
\end{eqnarray}
In the present context, where $T^*$ lies in a top-dimensional stratum, $\Phi_{T^*}(T)=\pi\circ\chi(\v)$.

Note that, when $T^*$ lies in a top-dimensional stratum, the map corresponding to $\Phi_{T^*}$ here obtained in \cite{BLO} in the case of $\TT$ was expressed as the composition of a similarly defined map on $Q_5$, a simpler auxiliary stratified space, with a map from $Q_5$ to $\TT$. Instead of the log map, that map on $Q_5$ was expressed in terms of the gradient of the squared distance function. The relationship between the latter and the log map shows that the resulting expression in \cite{BLO} is equivalent to the one defined here. The derivation of $\Phi_{T^*}$ from $\log_{T^*}$ implicitly requires that the tangent space to $\Tm$ at $T^*$, in which the image of $\log_{T^*}$ lies, be translated to the parallel copy $\mathbb{R}^m$ at the origin, in which it makes sense to add the coordinate vector $\t^*$. As a result, for all $\widetilde T^*$ in the same stratum as $T^*$, the image of $\Phi_{\widetilde T^*}$ will lie in this same subspace $\mathbb{R}^m$. 

\section{Fr\'echet means on a top-dimensional stratum}

Let $\mu$ be a probability measure on $\Tm$ and assume that the Fr\'echet function for $\mu$ is finite. The space $\Tm$ being $CAT(0)$ implies that the Fr\'echet function for $\mu$ is strictly convex so that, in particular, the Fr\'echet mean of $\mu$ is unique when it exists. In this section, we consider the case when this mean, denoted by $T^*$, lies in a top-dimensional stratum. For this, as in the previous section, we identify any tree $\widetilde T^*$ in the stratum of $\Tm$ in which $T^*$ lies with the point in the positive orthant of $\mathbb{R}^m$ having the lengths of the internal edges of $\widetilde T^*$ as coordinates in the canonical order. In particular $T^*=(t^*_1,\cdots,t^*_m)$. 

\vskip 6pt
First, we use the log map to give a necessary and sufficient condition for $T^*$ to be the Fr\'echet mean of $\mu$ as follows, generalising the characterisation of Fr\'echet means on complete and connected Riemannian manifolds of non-negative curvature. In particular, it shows that, when $T$ is a random variable on $\Tm$ with distribution $\mu$ and $T^*$ is in a top-dimensional stratum, then $T^*$ is the Fr\'echet mean of $\mu$ if and only if $T^*$ is the Euclidean mean of the Euclidean random variable $\Phi_{T^*}(T)$.

\begin{lemma}
Assume that the Fr\'echet mean $T^*$ of $\mu$ lies in a top-dimensional stratum. Then, $T^*$ is characterised by the following condition
\begin{eqnarray}
\int_{\Tm}\Phi_{T^*}(T)\,d\mu(T)=T^*.
\label{eqn4}
\end{eqnarray}
\label{lem1a}
\end{lemma}

\begin{proof}
It can be checked that, since $T^*$ lies in a top-dimensional stratum, the squared distance $d(T^*,T)^2$ is differentiable at $T^*$ and its gradient at $T^*$ is $-2\log_{T^*}(T)$. Thus, as discussed in \cite{BLO} $T^*$, lying in a top-dimensional stratum, is the Fr\'echet mean of a given probability measure $\mu$ on $\Tm$ if and only if
\[\int_{\Tm}\log_{T^*}(T)\,d\mu(T)=0.\]
Then, the required result follows by re-expressing the above condition for $T^*$ to be the Fr\'echet mean of $\mu$ in terms of $\Phi_{T^*}$ given by \eqref{eqn3a}.
\end{proof}

The derivation of the central limit theorem for Fr\'echet means in $\Tm$ requires the study of the change of $\Phi_{T^*}$ as $T^*$ changes with $T$ remaining fixed. For this we recall that, for a fixed $T$, the minimal support for the geodesic between $T^*$ and $T$ determines a particular maximal cell, in which $T$ lies, of the polyhedral subdivision with respect to $T^*$. When the minimal support for the geodesic between $\widetilde T^*$ and $T$ is the same as that for $T^*$ and $T$, we shall say that the two resulting maximal cells \textit{correspond to each other}. We have the following result on the derivative of $\Phi_{T^*}$ with respect to $T^*$, noting that the derivative of the map 
\[(x_1,\cdots,x_l)\mapsto\frac{1}{\|(x_1,\cdots,x_l)\|}(x_1,\cdots,x_l)\]
is
\[M^\dagger_{(x_1,\cdots,x_l)}=\frac{1}{\|(x_1,\cdots,x_l)\|}I_l-\frac{1}{\|(x_1,\cdots,x_l)\|^3}\begin{pmatrix}x_1\\\vdots\\x_l\end{pmatrix} \begin{pmatrix}x_1&\cdots&x_l\end{pmatrix},\]
where in particular, when $l=1$, $M^\dagger_{x_1}=0$.

\begin{lemma}
Assume that $T^*\in\Tm$ lies in a top-dimensional stratum. Then, for any fixed $T\in\Tm$ lying in the interior of a maximal cell of the polyhedral subdivision with respect to $T^*$, $\Phi_{T^*}(T)$ is differentiable with respect to $T^*$. Moreover, for such $T$, if $\v^*_i$ is as defined in \eqref{eqn2b} prior to Theorem $\ref{thm1}$ and $\v_i$ defined analogously for $i=1,\cdots,k$, where $k=k(T^*,T)$, then the derivative of $\Phi_{T^*}(T)$ at $T^*$, with respect to $T^*$, is given by
\begin{eqnarray}
M_{T^*}(T)=P^\top_{T^*\!,\,T}\hbox{\rm diag}\{v_1M^\dagger_{\v^*_1},\cdots, v_kM^\dagger_{\v^*_k}\}\,P_{T^*\!,\,T}
\label{eqn4a}
\end{eqnarray}
where $v_i=\|\v_i\|$ and $P_{T^*\!,\,T}$ denotes the matrix representing the permutation $\pi$ defined by \eqref{eqn1b}.
\label{lem2}
\end{lemma}

Note that, for the sub-matrix $v_iM^\dagger_{\v^*_i}$ to be non-zero, $\v^*_i$ must be at least 2-dimensional and, by definition, $v_i$ is non-positive so that $T$ must lie in $\mathcal{S}^{\phantom{A}}_{T^*}$. In particular, if $k(T^*,T)=m-|A_0|$, in other words, if the geodesic between $T^*$ and $T$ is `straight', then the derivative of $\Phi_{T^*}(T)$ at $T^*$ is zero. This could be seen directly: since, in that case, the tree space geodesic between $T^*$ and $T$ would be a Euclidean geodesic between them. Then, $\log_{T^*}(T)=\t-\t^*$ so that $\Phi_{T^*}(T)=\t$ independent of $\t^*$.

\begin{proof}
By the discussion preceding the lemma, the edges common to $T^*$ and $T$ will make no contribution to the derivative. Since the polyhedral subdivision is continuous with respect to $T^*$, it is sufficient to show that, when $\tilde T^*$ is sufficiently close to $T^*$, so that in particular $T^*$ and $\tilde T^*$ lie in the same top stratum and $T$ lies in the interior of the corresponding maximal cells of the polyhedral subdivisions with respect to $T^*$ and $\widetilde T^*$, we have
\begin{eqnarray}
\begin{array}{rcl}
&&\Phi_{\tilde T^*}(T)-\Phi_{T^*}(T)\\
&\approx&(\tilde T^*-T^*)\,P^\top_{T^*\!,\,T}\hbox{\rm diag}\{v_1M^\dagger_{\v^*_1},\cdots,v_kM^\dagger_{\v^*_k}\}\,P_{T^*\!,\,T}\\
&&+\|T\|\,o(\|\tilde T^*-T^*\|).
\end{array}
\label{eqn5}
\end{eqnarray}
To show \eqref{eqn5}, it is sufficient to assume that $T^*$ and $T$ have no common edge. Moreover, since $\pi^{\phantom{A}}_{T^*\!,\,T}$, and so $P_{T^*\!,\,T}$, is a linear map, its derivative is identical with itself. Hence, by applying the appropriate permutation to re-order the $\tilde\u^*$ and $\u$ corresponding to $\widetilde T^*$ and $T$ when necessary, it is sufficient to show that
\begin{eqnarray*}
&&\{\Phi_{\tilde T^*}(T)-\Phi_{T^*}(T)\}P^\top_{T^*\!,\,T}\\
&\!\!\!\!\approx&\!(\tilde\u^*-\u^*)\,\hbox{\rm diag}\{v_1M^\dagger_{\v^*_1},\cdots, v_kM^\dagger_{\v^*_k}\}+\|T\|\,o(\|\tilde T^*-T^*\|).
\end{eqnarray*}

Since $T$ lies in the interior of a maximal cell of the polyhedral subdivision of $\Tm$ with respect to $T^*$, then $v^*_i/v_i>v^*_{i+1}/v_{i+1}$ for all $i$, where all $v_i$ are negative. By continuity, all these strict inequalities hold when $v^*_i$ is replaced by $\tilde v^*_i$ if $\widetilde T^*$ is sufficiently close to $T^*$. Hence, $T$ lies in the interior of a maximal cell of the polyhedral subdivision of $\Tm$ with respect to $\widetilde T^*$. Thus, the only difference between the expressions for $\Phi_{T^*}(T)$ and $\Phi_{\widetilde T^*}(T)$ is that $\v^*$ and $\v^*_i$ in the former are replaced by $\tilde\v^*$ and $\tilde \v^*_i$ respectively in the latter. It follows that, in this case, the difference $\{\Phi_{\tilde T^*}(T)-\Phi_{T^*}(T)\}P^\top_{T^*\!,\,T}$ can be expressed as 
\[\left(v_1\left(\frac{1}{\|\tilde\v^*_1\|}\tilde\v^*_1-\frac{1}{\|\v^*_1\|}\v^*_1\right),\cdots,v_k\left(\frac{1}{\|\tilde\v^*_k\|}\tilde\v^*_k-\frac{1}{\|\v^*_k\|}\v^*_k\right)\right).\]
The required result follows by applying the first order Taylor expansion to each sub-vector component and using the formula preceding the statement of the Lemma.   
\end{proof}

If $T$ lies on the boundary of a maximal cell of the polyhedral subdivision of $\Tm$ with respect to $T^*$, each choice of maximal cell of the polyhedral subdivision with respect to $T^*$ will determine a support for the geodesic from $T^*$ to $T$. If we restrict the neighbouring $\widetilde T^*$ to move from $T^*$  in a direction such that $T$ lies in the corresponding maximal cell of the polyhedral subdivision with respect to $\widetilde T^*$, then the argument in the proof for Lemma \ref{lem2} still holds. Thus, $\Phi_{T^*}(T)$ will have all directional derivatives, at $T^*$, with respect to $T^*$ having similar forms to that given in Lemma \ref{lem2}. However, some different directions will require different choices of maximal cell of the polyhedral subdivision with respect to $T^*$ in which $T$ lies. Thus, the directional derivative will have different forms and $\Phi_{T^*}(T)$ will not be differentiable. 

\vskip 6pt
Lemma \ref{lem2} enables us to obtain the limiting distribution of the sample Fr\'echet means of a sequence of \textit{iid} random variables on $\Tm$ when the Fr\'echet mean of the underlying probability measure lies in a top-dimensional stratum as follows, recalling that $\mathcal{D}_{T^*}$, defined in Definition $\ref{def1}$, is the subset of $\Tm$ consisting of all trees that lie on the boundaries of maximal cells determined by the polyhedral subdivision with respect to $T^*$. On one hand, the result shows that, in this case, the limiting distribution, being a Gaussian distribution, bears a certain similarity to that of the sample means of Euclidean random variables. On the other hand, recalling that the derivative of $\Phi_{T^*}(T)$ at $T^*$ is zero if $T\not\in\mathcal{S}^{\phantom{A}}_{T^*}$, it also shows that the role played by $\mathcal{S}^{\phantom{A}}_{T^*}$ in the limiting behaviour of the sample Fr\'echet means is reflected in the covariance structure of the Gaussian distribution, departing from the limiting distribution of the sample means of Euclidean random variables. 

\begin{theorem}
Let $\mu$ be a probability measure on $\Tm$ with finite Fr\'echet function and with Fr\'echet mean $T^*$ lying in a top-dimensional stratum. Assume that $\mu\left(\mathcal{D}_{T^*}\right)=0$. Suppose that $\{T_i\,:\,i\geqslant1\}$ is a sequence of iid random variables in $\Tm$ with probability measure $\mu$ and denote by $\hat T_n$ the sample Fr\'echet mean of $T_1,\cdots,T_n$. Then,
\[\sqrt{n}(\hat T_n-T^*)\buildrel{d}\over\longrightarrow N(0,A^\top VA),\qquad\hbox{ as }n\rightarrow\infty,\]
where $V$ is the covariance matrix of the random variable $\log_{T^*}(T_1)$, or equivalently that of $\Phi_{T^*}(T_1)$, and  
\begin{eqnarray}
A=\left\{I-E\left[M_{T^*}(T_1)\right]\right\}^{-1},
\label{eqn5a}
\end{eqnarray}
assuming that this inverse exists, and where $M_{T^*}(T)$ is the $m\times m$ matrix defined by \eqref{eqn4a}.
\label{thm2}
\end{theorem} 

\begin{proof}
The main argument underlying the proof is similar to that of the proof in \cite{BLO} for $\TT$, i.e., to express the difference between the Fr\'echet mean of the underlying probability measure and the sample Fr\'echet means in terms of the difference $\Phi_{\tilde T^*}(T_i)-\Phi_{T^*}(T_i)$. However, the proof in \cite{BLO} relies on an explicit embedding that is only valid for $\TT$. As a consequence of Lemma \ref{lem2}, we can now achieve this for any tree space.  

Since $\hat T_n$ is the Fr\'echet sample mean of $T_1,\cdots,T_n$, then for sufficiently large $n$, $\hat T_n$ will be close to $T^*$ a.s. (cf. \cite{HZ}) and, in particular, lie in the same stratum as $T^*$. Thus, the above results \eqref{eqn4} and \eqref{eqn5} give
\begin{eqnarray*}
\sqrt{n}(\hat T_n-T^*)&=&\frac{1}{\sqrt n}\sum_{i=1}^n\{\Phi_{\hat T_n}(T_i)-T^*\}\\
&=&\frac{1}{\sqrt n}\sum_{i=1}^n\{\Phi_{T^*}(T_i)-T^*\}+\frac{1}{\sqrt n}\sum_{i=1}^n\{\Phi_{\hat T_n}(T_i)-\Phi_{T^*}(T_i)\}\\
&\approx&\frac{1}{\sqrt n}\sum_{i=1}^n\{\Phi_{T^*}(T_i)-T^*\}+\sqrt{n}(\hat T_n-T^*)\frac{1}{n}\sum_{i=1}^nM_{T^*}(T_i)\\
&&+o(\|\hat T_n-T^*\|)\frac{1}{\sqrt n}\sum_{i=1}^n\|T_i\|.
\end{eqnarray*}
Hence,
\begin{eqnarray*}
&&\sqrt{n}(\hat T_n-T^*)\left\{I-\frac{1}{n}\sum_{i=1}^nM_{T^*}(T_i)\right\}\\
&\approx&\frac{1}{\sqrt n}\sum_{i=1}^n\left\{\Phi_{T^*}(T_i)-T^*\right\}+o(\|\hat T_n-T^*\|)\frac{1}{\sqrt n}\sum_{i=1}^n\|T_i\|.
\end{eqnarray*} 
Since $\{\Phi_{T^*}(T_i)\,:\,i\geqslant1\}$ is a sequence of \textit{iid} random variables in $\mathbb{R}^m$ with mean $T^*$ and $\{M_{T^*}(T_i)\,:\,i\geqslant1\}$ is a sequence of \textit{iid} random matrices, the following theorem follows from the standard Euclidean result as in \cite{BLO}.
\end{proof}

\vskip 6pt
Recalling that $M_{T^*}(T_1)=0$ for $T_1$ not lying in the singularity set of $\log_{T^*}$, we see that the contribution to $E\left[M_{T^*}(T_1)\right]$ consists of all singular points of $\log_{T^*}$. For $m=1$, i.e. the case for $\T_{\!3}$, the only possible choice for $k$ is $k=1=m$ which implies that $M_{T^*}(T)\equiv0$, so that the above result for this special case is the same as that obtained in \cite{HHLMMMNOPS}. For $m=2$, i.e. the case for $\TT$, the only possible case for $T$ lying in the singularity set of $\log_{T^*}$ is when $k=1$, which corresponds to the geodesic between $T^*$ and $T$ passing through the origin and $\Phi_{T^*}(T)=-\|T\|\frac{1}{\sqrt{(t^*_1)^2+(t^*_2)^2}}(t^*_1,t^*_2)$. Then, the corresponding $M_{T^*}(T)$ has the expression
\[M_{T^*}(T)=-\|T\|\frac{1}{\|T^*\|^3}\begin{pmatrix}-t^*_2\\t^*_1\end{pmatrix}\begin{pmatrix}-t^*_2&t^*_1\end{pmatrix},\]
so that the above result for this case recovers that in \cite{BLO}.

\vskip 6pt
Note that $\mu$ induces, by $\log_{T^*}$, a probability distribution $\mu'$ on the tangent space of $\Tm$ at $T^*$. Then, the sample Fr\'echet means of $\mu'$ are the standard Euclidean means 
\[\frac{1}{n}\sum_{i=1}^n\log_{T^*}(T_i)=\frac{1}{n}\sum_{i=1}^n\left\{\Phi_{T^*}(T_i)-T^*\right\},\] 
so that the rescaled sample Fr\'echet means have the limiting distribution $N(0,V)$. However, the sample Fr\'echet means of $\mu'$ are generally different from $\log_{T^*}(\hat T_n)$, the log images of the sample Fr\'echet means of $\mu$, and there is no closed expression for the relationship between the two.

\vskip 6pt
It is also interesting to compare the result of Theorem \ref{thm2} with the limiting distributions for the sample Fr\'echet means on Riemannian manifolds obtained in \cite{KL}. Both limiting distributions take a similar form, with the role played by curvature in the case of manifolds being replaced here by the global topological structure of the tree space.

\section{Fr\'echet means on a stratum of co-dimension one}

A stratum $\O(\Sigma)$ of co-dimension one corresponding to the set $\Sigma$ of mutually compatible edge-types arises as a boundary face of a top-dimensional stratum when one, and only one, internal edge of the latter is given length zero so that its two vertices are coalesced to form a new vertex of valency four. The four incident edges determine disjoint subsets $A,B,C,X$ of leaves and root, where $X$ contains the root. Then an additional internal edge may be introduced to $\Sigma$, namely $\alpha$, $\beta$ or $\gamma$ that correspond respectively to the sets of leaves $A\cup B$, $A\cup C$ or $B\cup C$. This gives top-dimensional strata $\O(\Sigma\cup\alpha)$, $\O(\Sigma\cup\beta)$ or $\O(\Sigma\cup\gamma)$, all of whose boundaries contain the stratum $\O(\Sigma)$. Moreover, these are the only such top-dimensional strata. For example, in Figure~\ref{fig:3_quadrants}, the leaves and root subsets are $A = \{a,b\}$, $B = \{c\}$, $C = \{d\}$, and $X = \{r\}$, while the sets of edge-types are $\Sigma = \{ \{a,b\}\}$, $\alpha = \{\{c,d\}\}$, $\beta = \{\{a,b,d\}\}$ and $\gamma= \{ \{a,b,c\}\}$.  

If $A>B>C$ is the canonical order of the sets of leaves, then $\alpha<\beta<\gamma$ is the induced order of the edges and corresponding semi-axes and, if we write the coordinates of a tree $T^*$ in $\O(\Sigma)$ as $(t^*_2,\cdots,t^*_m)$, we can write the coordinates of trees in the neighbouring orthants as $(t^*_\alpha,t^*_\beta,t^*_\gamma,t^*_2,\cdots,t^*_m)$ where precisely two of $t^*_\alpha,t^*_\beta$ and $t^*_\gamma$ are zero, since the remaining $m-1$ edge-types are common to all the trees involved in these three orthants and their common boundary component. Note however that, although the coordinates $(t^*_\alpha,t^*_\beta,t^*_\gamma)$ and $(t^*_2,\cdots,t^*_m)$ can be chosen in canonical order, that will not in general be the case for the full set of coordinates.

It is clear now that the tree space $\Tm$ is not locally a manifold at any tree in the strata of co-dimension one. However, the stratification enables us to define, at a tree in a stratum of positive co-dimension, its tangent cone (cf. \cite{BH}) to consist of all initial tangent vectors of smooth curves \textit{starting} from that tree. Then, the tangent cone to $\Tm$ at a tree in a stratum of co-dimension one is an open book (cf. \cite{HHLMMMNOPS}) with three pages extending each of the three strata and with the stratum of co-dimension one in which the tree lies being extended to form its spine. 

\vskip 6pt
The definition of the log map \eqref{eqn1} applies equally to a tree $T^*$ in a stratum $\sigma$ of co-dimension one: if the geodesic from $T^*$ to $T$ passes through one of the three strata whose boundary includes $\sigma$, the unit vector component of $\log_{T^*}(T)$ is taken in the same direction in the page of the tangent book that corresponds to that stratum. The scalar component of the log map is still the distance between the trees. Similarly, the definition \eqref{eqn3a} for $\Phi_{T^*}$ remains valid in this case. 

From now on, we assume that $T^*$ lies in a stratum $\O(\Sigma)$ of co-dimension one. Although the squared distance $d(T^*,T)^2$ is no longer differentiable at $T^*$, it has directional derivatives along all possible directions. Hence, the condition for $T^*$ to be the Fr\'echet mean of a probability measure $\mu$ on $\Tm$, i.e. the condition for $T^*$ to satisfy
\[\int_{\Tm} d(T^*,T)^2\,d\mu(T)<\int_{\Tm} d(T',T)^2\,d\mu(T)\qquad\hbox{ for any }T'\not=T^*,\]
becomes that the Fr\'echet function for $\mu$ has, at $T^*$, non-negative directional derivatives along all possible directions. To investigate the latter condition, we label the three strata joined at the stratum $\O(\Sigma)$, of co-dimension one, in which $T^*$ lies as the $\alpha$-, $\beta$- and $\gamma$-strata and denote by $\log^\alpha_{T^*}$, $\log^\beta_{T^*}$ and $\log^\gamma_{T^*}$ respectively the modifications of the map $\log_{T^*}$ that agree with $\log_{T^*}$ on the domains for which the image lies in the pages of the tangent book tangent to the $\alpha$-, $\beta$- and $\gamma$-strata respectively and are zero elsewhere. That is, for example,
\[\log^\alpha_{T^*}(T)=\left\{
   \begin{array}{ll}\log_{T^*}(T)&\hbox{ if $T$ is such that $\log_{T^*}(T)$ lies in the page of the}\\&\quad\hbox{ tangent book tangent to the }\alpha\hbox{-orthant}\\0&\hbox{ otherwise}.\end{array} 
   \right.\] 
Write $\e_\alpha$, $\e_\beta$ and $\e_\gamma$ for the outward unit vectors in the tangent book at $T^*$ lying in the page tangent to the $\alpha$-, $\beta$- and $\gamma$-strata respectively and orthogonal to its spine, and define 
\[I_i=\int_{\Tm}\langle\log^i_{T^*}(T),\,\e_i\rangle\,d\mu(T),\qquad i=\alpha,\beta,\gamma.\]
We also define $\log_{T^*}^s$ to be the modification of $\log_{T^*}$ with respect to the spine of the tangent book, the tangent space to $\O(\Sigma)$, analogous to the above $\log^i_{T^*}$. Then, we have the following characterisation of $T^*$ in a stratum of co-dimension one to be the Fr\'echet mean of $\mu$, in terms of the derivatives of the Fr\'echet function along the three directions orthogonal to the tangent space to $\O(\Sigma)$, as well as the Euclidean mean of $\log_{T^*}^s(T)$, where $T$ is a random variable on $\Tm$ with distribution $\mu$.

\begin{lemma}
With the notation and definition above, a given tree $T^*$ in a stratum $\O(\Sigma)$ of co-dimension one is the Fr\'echet mean of a given probability measure $\mu$ on $\Tm$ if and only if
\begin{eqnarray}
I_\alpha\leqslant I_\beta+I_\gamma,\qquad I_\beta\leqslant I_\gamma+I_\alpha,\qquad I_\gamma\leqslant I_\alpha+I_\beta
\label{eqn7}
\end{eqnarray}
and
\begin{eqnarray}
\int_{\Tm}\log^s_{T^*}(T)\,d\mu(T)=0.
\label{eqn7a}
\end{eqnarray}
\label{lem3}
\end{lemma}

\begin{proof}
Recall that since $T^*\in\O(\Sigma)$, the condition for $T^*$ to be the Fr\'echet mean of a probability measure $\mu$ on $\Tm$ is that the Fr\'echet function for $\mu$ has, at $T^*$, non-negative directional derivatives along all possible directions. 

For any vector $\w$ at $T^*$ which is tangent to $\O(\Sigma)$, the non-negativity of the directional derivative along $\w$ can be expressed as $\int_{\Tm}\langle\log^s_{T^*}(T),\,\w\rangle\,d\mu(T)\leqslant0$. Since $-\w$ also tangent to $\O(\Sigma)$ at $T^*$, this inequality must be an equality for all such $\w$, which gives \eqref{eqn7a}. Hence, by linearity, the non-negativity of directional derivatives, of the Fr\'echet function for $\mu$, at $T^*$ along all possible directions may be characterised by requiring the non-negativity of the directional derivatives along the $\e_\alpha$, $\e_\beta$ and $\e_\gamma$ directions, together with \eqref{eqn7a}. However, analogously to the deduction in \cite{BLO}, it can be checked that the requirement for the directional derivative along each of the $\e_\alpha$, $\e_\beta$ and $\e_\gamma$ directions to be non-negative is respectively equivalent to each of the inequalities \eqref{eqn7}. 
\end{proof}

To see the relation between the inequalities \eqref{eqn7} and the asymptotic behaviour of sample Fr\'echet means, we will use a folding map $F_\alpha$ (cf. \cite{HHLMMMNOPS}) that operates on the tangent book at $T^*$. The map $F_\alpha$ folds the two pages that are tangent to the $\beta$- and $\gamma$-strata onto each other, so that they form the complement in $\mathbb{R}^m$ of the closure of the page tangent to the $\alpha$-stratum. Define $F_\beta$ and $F_\gamma$ similarly. Then, $F_\alpha\circ\log_{T^*}$ maps $\Tm$ to $\mathbb{R}^m$ and, in fact, is the limit of $\log_{\widetilde T^*}$ when $\widetilde T^*$ tends to $T^*$ from the $\alpha$-stratum. In addition, we modify the definition \eqref{eqn2} of $\chi^{\phantom{A}}_{T^*\!,\,T}(\e_i)$ to be $\pi^{-1}_{T^*\!,\,T}(\e_\alpha)$ when, and only when, the $\v^*_i$ in \eqref{eqn2} contains $t^*_\alpha$ and is 1-dimensional. With this modification and by noting that the argument leading to Lemma \ref{lem1}, as well as its result, still hold when $T^*$ lies in a stratum of co-dimension one, the results of Theorem \ref{thm1} and Lemma \ref{lem2} can be extended to obtain the expression for $F_\alpha\circ\log_{T^*}$ and its derivative, and the analogues with $\beta$ or $\gamma$ replacing $\alpha$, when the necessary care is taken of which stratum is to contain the initial geodesic. Moreover,
\begin{eqnarray}
\int_{\Tm}\langle F_\alpha\circ\log_{T^*}(T),\e_\alpha\rangle\,d\mu(T)=I_\alpha-I_\beta-I_\gamma.
\label{eqn7b}
\end{eqnarray}
These observations lead to the following lemma which extends the results obtained in \cite{HHLMMMNOPS} for open books and in \cite{BLO} for $\TT$ and relates the Fr\'echet means of large samples avoiding a stratum to the strict-positivity of the derivative of the Fr\'echet function along the corresponding orthogonal direction to the tangent space to $\O(\Sigma)$.

\begin{lemma}
Let $T^*$ be the Fr\'echet mean of a given probability measure $\mu$ on $\Tm$, and lie in a stratum $\O(\Sigma)$ of co-dimension one. Assume that $\mu\left(\mathcal{D}_{T^*}\right)=0$, where $\mathcal{D}_{T^*}$ is defined in Definition $\ref{def1}$, and that, at $T^*$, $I_\alpha<I_\beta+I_\gamma$. If $\{T_i\,:\,i\geqslant1\}$ is a sequence of iid random variables in $\Tm$ with probability measure $\mu$ then, for all sufficiently random large $n$, the sample Fr\'echet mean $\hat T_n$ of $T_1,\cdots,T_n$ cannot lie in the $\alpha$-stratum. 
\label{lem4}
\end{lemma}

\begin{proof}
Since $\hat T_n$ converges to $T^*$ a.s. as $n$ tends to infinity (cf.\cite{HZ}) we only need to show that, for all sufficiently large $n$, $\hat T_n$ cannot lie in the neighbourhood of $T^*$, restricted to the $\alpha$-stratum.

Consider the probability measure $\mu_\alpha$ induced from $\mu$ by $F_\alpha\circ\log_{T^*}$ on the Euclidean space. Then, under the given conditions, it follows from \eqref{eqn7b} that the Euclidean mean of $\mu_\alpha$ lies on the open half of the Euclidean space complement to the page tangent to the $\alpha$-stratum (cf. also \cite{HHLMMMNOPS}). Thus, for all sufficiently large $n$, the Euclidean mean of the induced random variables $F_\alpha\circ\log_{T^*}(T_1),\cdots,F_\alpha\circ\log_{T^*}(T_n)$,
\[\hat T^\alpha_n=\frac{1}{n}\sum_{i=1}^nF_\alpha\circ\log_{T^*}(T_i),\]
does not lie in the closed half of this Euclidean space where the page tangent to the $\alpha$-stratum lies. This implies that, for all sufficiently large $n$, 
\begin{eqnarray}
\langle\hat T^\alpha_n,\e_\alpha\rangle<0.
\label{eqn8}
\end{eqnarray}

If it were possible that, for arbitrarily large $n$, $\hat T_n$ lies in the $\alpha$-stratum, we could obtain a contradiction. Firstly, noting the observations prior to the lemma and following the arguments of the proof for Lemma \ref{lem2}, for all sufficiently large $n$, we have
\begin{eqnarray}
\begin{array}{rcl}
\dfrac{1}{n}\displaystyle\sum_{i=1}^n\Phi_{\hat T_n}(T_i)
&=&\dfrac{1}{n}\displaystyle\sum_{i=1}^nF_\alpha\circ\Phi_{T^*}(T_i)
+\,\,(\hat T_n-T^*)\dfrac{1}{n}\displaystyle\sum_{i=1}^nM_{T^*}(T_i)\\
&&+\,\,o(\|\hat T_n-T^*\|)\dfrac{1}{n}\displaystyle\sum_{i=1}^n\|T_i\|,
\end{array}
\label{eqn9}
\end{eqnarray}
where $M_{T^*}(T)$ is given by \eqref{eqn4a} and $F_\alpha\circ\Phi_{T^*}=F_\alpha\circ\log_{T^*}+T^*$. However, on the one hand, since $\displaystyle\frac{1}{n}\sum_{i=1}^n\Phi_{\hat T_n}(T_i)$ $=\hat T_n$ and since $\hat T_n$ lies in the $\alpha$-stratum, $\langle\hat T_n,\e_\alpha\rangle>0$, so that 
\begin{eqnarray}
\left\langle\frac{1}{n}\sum_{i=1}^n\Phi_{\hat T_n}(T_i),\,\e_\alpha\right\rangle>0.
\label{eqn10}
\end{eqnarray}
While on the other hand, it follows from $\langle T^*,\e_\alpha\rangle=0$ and from \eqref{eqn8} that 
\begin{eqnarray}
\left\langle\frac{1}{n}\sum_{i=1}^nF_\alpha\circ\Phi_{T^*}(T_i),\,\e_\alpha\right\rangle=\langle\hat T^\alpha_n,\e_\alpha\rangle<0.
\label{eqn11}
\end{eqnarray}
It can also be checked that  
\[M_{T^*}(T_i)\,\e_\alpha=\frac{v^\alpha_i}{\|\v^*_{i,s}\|}\e_\alpha,\]
where $v^\alpha_i=v_{i,s}$, if $t_\alpha$ corresponds to a coordinate of $\v^*_{i,s}$ and if the dimension of $\v^*_{i,s}$ is greater than one, and $v^\alpha_i=0$ otherwise. Then, since $v^\alpha_i\leqslant0$, for each $i$ 
\begin{eqnarray}
\langle(\hat T_n-T^*)\,M_{T^*}(T_i),\,\e_\alpha\rangle=v^\alpha_i\langle\hat T_n,\e_\alpha\rangle\leqslant0.
\label{eqn12}
\end{eqnarray}
Equations \eqref{eqn11} and \eqref{eqn12} together imply that, for all sufficiently large $n$, the $e_\alpha$-component of the right hand side of \eqref{eqn9} is negative, which contradicts \eqref{eqn10}.
\end{proof}

With the result of Lemma \ref{lem4}, we now have the limiting distribution of the sample Fr\'echet means on $\Tm$ given by the next theorem, which is the generalisation of the result for $\TT$ given in Theorem 2 in \cite{BLO}. In particular, it shows that the limiting distribution can take any of four possible forms, all related to a Gaussian distribution, depending on the number of the strictly positive derivatives of the Fr\'echet function along the three directions orthogonal to the tangent space to $\O(\Sigma)$. For clarity, we have assumed in the following that the coordinates $(t_i,t_2,\cdots,t_m)$, $i=\alpha,\beta,\gamma$, discussed at the beginning of the section are all in the canonical order, so that they give the coordinates for trees in each of the three strata. Otherwise, a further permutation of the coordinates, which we have suppressed, will be necessary to bring them into canonical order and so to validate the result.

\begin{theorem}
Let $T^*$ in a stratum $\O(\Sigma)$ of co-dimension one be the Fr\'echet mean of a given probability measure $\mu$ on $\Tm$. Assume that $\mu\left(\mathcal{D}_{T^*}\right)=0$, where $\mathcal{D}_{T^*}$ is defined in Definition $\ref{def1}$. Let further $\{T_i\,:\,i\geqslant1\}$ be a sequence of iid random variables in $\Tm$ with probability measure $\mu$ and write $\hat T_n$ for the sample Fr\'echet mean of $T_1,\cdots,T_n$.  
\begin{enumerate}
\item[$(a)$] If all three inequalities in \eqref{eqn7} are strict then, for all sufficiently large $n$, $\hat T_n$ will lie in the stratum $\O(\Sigma)$ and the sequence $\sqrt{n}\{(\hat t_2^n,\cdots,\hat t_m^n)-(t^*_2,\cdots,t^*_m)\}$ of the coordinates of $\sqrt{n}\{\hat T_n-T^*\}$ on the spine will converge in distribution to $N(0,A^\top_sV_sA_s)$ as $n\rightarrow\infty$, where $V_s$ is the covariance matrix of the random variable $\log^s_{T^*}(T_1)$, $A_s=P_s^\top AP_s$, $P_s$ is the projection matrix to the subspace of $\mathbb{R}^m$ with the first coordinate removed and $A$ is as given in $\eqref{eqn5a}$.
\item[$(b)$] If the first inequality in \eqref{eqn7} is an equality and the other two are strict then, for all sufficiently large $n$, $\hat T_n$ will lie in the $\alpha$-stratum and
\[\sqrt{n}\{\hat T_n-T^*\}\buildrel{d}\over\longrightarrow(\max\{0,\eta_1\},\eta_2,\cdots,\eta_m),\quad\hbox{as }n\rightarrow\infty,\] 
where $(\eta_1,\cdots,\eta_m)\sim N(0,A^\top V A)$, $V$ is the covariance matrix of $F_\alpha\circ\log_{T^*}(T_1)$ and $A$ is as in \eqref{eqn5a} with $t^*_1=0$. 
\item[$(c)$] If the first two inequalities in \eqref{eqn7} are equalities and the third is strict then, for all sufficiently large $n$, $\hat T_n$ will lie either in the $\alpha$-stratum or in the $\beta$-stratum and the limiting distribution of $\sqrt{n}\{\hat T_n-T^*\}$, as $n\rightarrow\infty$, will take the same form as that of $(\eta_1,\cdots,\eta_m)$ above, where the coordinates of $\hat T_n$ are taken as $(\hat t_\alpha^n,\hat t_2^n,\cdots,\hat t_m^n)$, respectively $(-\hat t_\beta^n,\hat t_2^n,\cdots,\hat t_m^n)$, if $\hat T_n$ is in the $\alpha$-stratum, respectively the $\beta$-stratum.
\item[$(d)$] If all the equalities in \eqref{eqn7} are actually equalities, then we have the same result as in $(a)$.
\end{enumerate}
\label{thm3}
\end{theorem} 

\begin{proof}
$(a)$ By Lemma \ref{lem4}, when $n$ is sufficiently large, $\hat T_n$ must lie in the stratum $\O(\Sigma)$ of co-dimension one so that it has zero first coordinate, i.e. $\hat T_n=(0,\hat t^n_2,\cdots,\hat t^n_m)$. Noting that $F_\alpha\circ\log^s_{\hat T_n}=\log^s_{\hat T_n}$, the result \eqref{eqn7a} of Lemma \ref{lem3} shows that $\hat t^n_i$, $i=2,\cdots,m$, are the respective coordinates of $\dfrac{1}{n}\displaystyle\sum_{i=1}^n F_\alpha\circ\Phi_{\hat T_n}(T_i)$, the sample Euclidean mean of $F_\alpha\circ\Phi_{\hat T_n}(T_1),\cdots,F_\alpha\circ\Phi_{\hat T_n}(T_n)$. Then a modification of the proof of Theorem \ref{thm2} to restrict it to the relevant coordinates of $\{F_\alpha\circ\Phi_{\hat T_n}(T_i):i\geqslant1\}$ gives the required limiting distribution of $\sqrt{n}\{(\hat t_2^n,\cdots,\hat t_m^n)-(t^*_2,\cdots,t^*_m)\}$. 

\vskip 6pt
$(b)$ We deduce from the assumed strict inequalities, from \eqref{eqn7a} and \eqref{eqn7b} and from Lemma \ref{lem4} that $T^*$ is the Euclidean mean of $F_\alpha\circ\Phi_{T^*}(T_1)$ and that, when $n$ is sufficiently large, $\hat T_n$ can only lie in the closure of the $\alpha$-stratum, so that it has coordinates $\hat T_n=(\hat t_\alpha^n,\hat t_2^n,\cdots,\hat t^n_m)$. 

Write
\begin{eqnarray}
\tilde F_\alpha\circ\Phi_{\hat T_n}(T)=\left\{
   \begin{array}{ll}\Phi_{\hat T_n}(T)&\hbox{if }\hat t_\alpha^n>0\\F_\alpha\circ\Phi_{\hat T_n}(T)&\hbox{if }\hat t_\alpha^n=0.\end{array} \right.
\label{eqn13}
\end{eqnarray}
Then, $\tilde F_\alpha\circ\Phi_{\hat T_n}(T)$ lies in $\mathbb{R}^m$ and, by \eqref{eqn7a}, $\hat t_j^n$, $j=2,\cdots,m$, are the respective coordinates of $\dfrac{1}{n}\displaystyle\sum_{i=1}^n\tilde F_\alpha\circ\Phi_{\hat T_n}(T_i)$. To see relationship between $\dfrac{1}{n}\displaystyle\sum_{i=1}^n\tilde F_\alpha\circ\Phi_{\hat T_n}(T_i)$ and $\hat t_\alpha^n$, we note that, if $\hat t^n_\alpha>0$,
\begin{eqnarray}
\frac{1}{n}\sum_{i=1}^n\tilde F_\alpha\circ\Phi_{\hat T_n}(T_i)=\frac{1}{n}\sum_{i=1}^n\Phi_{\hat T_n}(T_i)=\hat T_n,
\label{eqn13a}
\end{eqnarray}
where the first equality follows from the definition of 
$\tilde F_\alpha\circ\Phi_{\hat T_n}(T)$ and the second follows from Lemma \ref{lem1a} as $\hat T_n$ lies in a top-dimensional stratum. Hence,
\[\left\langle\frac{1}{n}\sum_{i=1}^n\tilde F_\alpha\circ\Phi_{\hat T_n}(T_i),\,\e_\alpha\right\rangle=\langle\hat T_n,\e_\alpha\rangle=\hat t^n_\alpha.\]
On the other hand, if $\hat t^n_\alpha=0$, then $\hat T_n$ lies in $\O(\Sigma)$ and
\[\frac{1}{n}\sum_{i=1}^n\tilde F_\alpha\circ\Phi_{\hat T_n}(T_i)=\frac{1}{n}\sum_{i=1}^n F_\alpha\circ\Phi_{\hat T_n}(T_i).\]
Applying Lemma \ref{lem3} and \eqref{eqn7b} to the empirical distribution centred on $T_1,\cdots,T_n$ with equal weights $1/n$, we also have
\[\left\langle\frac{1}{n}\sum_{i=1}^n\tilde F_\alpha\circ\Phi_{\hat T_n}(T_i),\,\e_\alpha\right\rangle\leqslant0.\]
Thus,
\[\hat t^n_\alpha=\max\left\{0,\,\,\left\langle\frac{1}{n}\sum_{i=1}^n\tilde F_\alpha\circ\Phi_{\hat T_n}(T_i),\,\e_\alpha\right\rangle\right\}.\]

Now, similarly to the proofs of Theorem \ref{thm2} and Lemma \ref{lem4}, the observations prior to Lemma \ref{lem4} imply that
\begin{eqnarray}
&&\frac{1}{\sqrt n}\sum_{i=1}^n\left\{\tilde F_\alpha\circ\Phi_{\hat T_n}(T_i)-T^*\right\}\nonumber\\
&=&\frac{1}{\sqrt n}\sum_{i=1}^n\left\{F_\alpha\circ\Phi_{T^*}(T_i)-T^*\right\}+\frac{1}{\sqrt n}\sum_{i=1}^n\left\{\tilde F_\alpha\circ\Phi_{\hat T_n}(T_i)-F_\alpha\circ\Phi_{T^*}(T_i)\right\}\nonumber\\
&\approx&\frac{1}{\sqrt n}\sum_{i=1}^n\left\{F_\alpha\circ\Phi_{T^*}(T_i)-T^*\right\}+\frac{1}{\sqrt{n}}(\hat T_n-T^*)\sum_{i=1}^nM_{T^*}(T_i)\label{eqn14}\\
&&+o(\|\hat T_n-T^*\|)\frac{1}{\sqrt n}\sum_{i=1}^n\|T_i\|,\nonumber
\end{eqnarray} 
where $M_{T^*}(T)$ is given by \eqref{eqn4a}. Since the first coordinate of $T^*$ is zero, so too are the entries, except for the diagonal one, in the first row and column of $M_{T^*}(T)$ and so also are the corresponding entries in the matrix $A$. Moreover, noting the comments following Lemma \ref{lem2} and the definition of $M^\dagger$ prior to that lemma, we see that the first diagonal entry of $M_{T^*}(T)$ is always non-positive. Thus, the first diagonal entry of $A$ must be positive, so that this is also the case for $\left\{I-\dfrac{1}{n}\displaystyle\sum_{i=1}^nM_{T^*}(T_i)\right\}^{-1}$, when $n$ is sufficiently large. 

Thus, when $\hat t_\alpha^n>0$, it follows from \eqref{eqn13a} and \eqref{eqn14} that
\begin{eqnarray*}
\sqrt{n}(\hat T_n-T^*)
&\approx&\frac{1}{\sqrt n}\sum_{i=1}^n\left\{F_\alpha\circ\Phi_{T^*}(T_i)-T^*\right\}\left\{I-\frac{1}{n}\sum_{i=1}^nM_{T^*}(T_i)\right\}^{-1}\\
&&+o(\|\hat T_n-T^*\|)\frac{1}{\sqrt n}\sum_{i=1}^n\|T_i\|.
\end{eqnarray*} 
In particular, for all sufficiently large $n$, the first coordinate of the random vector given by the first term on the right is positive. When $\hat t_\alpha^n=0$ the above approximation still holds except for the first coordinate. In that case, $\langle(\hat T_n-T^*)M_{T^*}(T_i),\,\e_\alpha\rangle=0$, following from the form of $M_{T^*}(T)$ noted above, and by \eqref{eqn14}, for sufficiently large $n$,  
\[\left\langle\frac{1}{\sqrt n}\sum_{i=1}^n\left\{F_\alpha\circ\Phi_{T^*}(T_i)-T^*\right\},\,\e_\alpha\right\rangle\leqslant0\]
up to higher order terms, which is equivalent to
\[\left\langle\frac{1}{\sqrt n}\sum_{i=1}^n\left\{F_\alpha\circ\Phi_{T^*}(T_i)-T^*\right\}\left\{I-\frac{1}{n}\sum_{i=1}^nM_{T^*}(T_i)\right\}^{-1},\,\e_\alpha\right\rangle\leqslant0\]
up to higher order terms. Hence, for sufficiently large $n$, we have
\[\sqrt{n}(\hat T_n-T^*)
\approx(\max\{0,\eta^n_1\},\eta^n_2,\cdots,\eta^n_m),\]
where
\[(\eta^n_1,\eta^n_2,\cdots,\eta^n_m)=\frac{1}{\sqrt n}\sum_{i=1}^n\left\{F_\alpha\circ\Phi_{T^*}(T_i)-T^*\right\}\left\{I-\frac{1}{n}\sum_{i=1}^nM_{T^*}(T_i)\right\}^{-1},\]
so that the required result follows from a similar argument to that of the proof for Theorem \ref{thm2}. 

\vskip 6pt
$(c)$ In this case, it follows from Lemma \ref{lem3} that $T^*$ is the Euclidean mean both of $F_\alpha\circ\Phi_{T^*}(T_1)$ and of $F_\beta\circ\Phi_{T^*}(T_1)$. Moreover, the integral $I_\gamma$ becomes zero and so, since the integrand is non-negative, the support of the measure on the tangent book at $T^*$ induced by $\mu$ is contained in the union of the leaves tangent to the $\alpha$- and $\beta$-strata together with the spine. 

It is now more convenient to represent the union of the $\alpha$- and $\beta$-strata by coordinates in the two orthants $\{(t_1,\cdots,t_m):t_2,\cdots,t_m\geqslant0\}$ of $\mathbb{R}^m$. For this, we map:
\[(t_\alpha,t_2,\cdots,t_m)\mapsto(t_\alpha,t_2,\cdots,t_m)\hbox{ and } 
(t_\beta,t_2,\cdots,t_m)\mapsto(t_\beta,t_2,\cdots,t_m)R,\]
where $R=\hbox{diag}\{-1,I_{m-1}\}$. Similarly, we define maps $\tilde\Phi_{(t_1,\cdots,t_m)}(T)$ to accord with this by $\tilde\Phi_{(-t_\beta,t_2,\cdots,t_m)}(T)=\Phi_{(t_\beta,t_2,\cdots,t_m)}(T)R$, while $\tilde\Phi_{(t_\alpha,t_2,\cdots,t_m)}=\Phi_{(t_\alpha,t_2,\cdots,t_m)}$. Since
$\Phi_{(0_\alpha,t_2,\cdots,t_m)}(T)=\Phi_{(0_\beta,t_2,\cdots,t_m)}(T)R$, the map $\tilde\Phi$ is indeed a.s. well defined for points $(0,t_2,\cdots,t_m)$. Clearly,
\[\tilde\Phi_{(t_1,t_2\cdots,t_m)}(T)=\left\{\begin{array}{ll}
   \tilde F_\alpha\circ\Phi_{(t_1,t_2,\cdots,t_m)}(T)&\hbox{if }t_1\geqslant0\\ \tilde F_\beta\circ\Phi_{(-t_1,t_2,\cdots,t_m)}(T)R&\hbox{if }t_1\leqslant0\end{array}\right.\]
where $\tilde F_\alpha$, similarly $\tilde F_\beta$, is defined by \eqref{eqn13}.
Under this new coordinate system, since $F_\alpha\circ\Phi_{T^*}(T_1)=F_\beta\circ\Phi_{T^*}(T_1)R$ a.s., we have in particular that 
\begin{eqnarray}
T^*=\int_{\Tm}\tilde\Phi_{(0,t_2^*,\cdots,t_m^*)}(T)\,d\mu(T).
\label{eqn15}
\end{eqnarray}

By Lemma \ref{lem4}, the given assumption also implies that, for sufficiently large $n$, $\hat T_n$ will a.s. lie either in the $\alpha$-stratum or in the $\beta$-stratum. If $\hat T_n$ lies in the $\alpha$-stratum, then $\hat t_\alpha^n>0$ and
\begin{eqnarray}
(\hat t_\alpha^n,\hat t_2^n,\cdots,\hat t_m^n)=\frac{1}{n}\sum_{i=1}^n\Phi_{\hat T_n}(T_i)=\frac{1}{n}\sum_{i=1}^n\tilde\Phi_{(\hat t_\alpha^n,\hat t_2^n,\cdots,\hat t_m^n)}(T_i)
\label{eqn16}
\end{eqnarray}
and, if $\hat T_n$ lies in the $\beta$-stratum with (original) coordinates $\hat T_n=(\hat t_\beta^n,\hat t_2^n,\cdots,\hat t_m^n)$, then
\begin{eqnarray}
(-\hat t_\beta^n,\hat t_2^n,\cdots,\hat t_m^n)=\frac{1}{n}\sum_{i=1}^n\Phi_{\hat T_n}(T_i)R=\frac{1}{n}\sum_{i=1}^n\tilde\Phi_{(-\hat t_\beta^n,\hat t_2^n,\cdots,\hat t_m^n)}(T_i).
\label{eqn17}
\end{eqnarray}
If $\hat T_n$ lies on the stratum $\O(\Sigma)$ of co-dimension one then, by applying the argument in $(b)$ to both $\hat t^n_\alpha=0$ and $\hat t^n_\beta=0$, we also have
\begin{eqnarray}
(0,\hat t_2^n,\cdots,\hat t_m^n)=\frac{1}{n}\sum_{i=1}^n\Phi_{(0_\alpha,\hat t_2^n,\cdots,\hat t_m^n)}(T_i)=\frac{1}{n}\sum_{i=1}^n\tilde\Phi_{(0,\hat t_2^n,\cdots,\hat t_m^n)}(T_i)\quad\hbox{ a.s.}.
\label{eqn18}
\end{eqnarray}
Recalling that, under the new coordinate system,
\begin{eqnarray*}
\hat T_n\equiv\left\{\begin{array}{ll}
   (\hat t_\alpha^n,\hat t_2^n,\cdots,\hat t_m^n)&\hbox{if $\hat T_n$ is in the $\alpha$-stratum}\\(-\hat t_\beta^n,\hat t_2^n,\cdots,\hat t_m^n)&\hbox{if $\hat T_n$ is in the $\beta$-stratum}\end{array}\right.
\end{eqnarray*}   
we have by \eqref{eqn15}, \eqref{eqn16}, \eqref{eqn17} and \eqref{eqn18} that, in terms of the new coordinates,
\begin{eqnarray*}
\sqrt{n}\{\hat T_n-T^*\}&=&\frac{1}{\sqrt n}\sum_{i=1}^n\left\{\tilde\Phi_{T^*}(T_i)-(0,t_2^*,\cdots,t_m^*)\right\}\\
&+&\frac{1}{\sqrt n}\sum_{i=1}^n\left\{\tilde\Phi_{\hat T_n}(T_i)-\tilde\Phi_{T^*}(T_i)\right\}.
\end{eqnarray*}
Hence, since \eqref{eqn14} still holds under this new coordinate system when $\tilde F_\alpha\circ\Phi_{\hat T_n}$ and $F_\alpha\circ\Phi_{T^*}$ there are replaced by $\tilde\Phi_{\hat T_n}$ and $\tilde\Phi_{T^*}$ respectively, a similar argument to that of the proof for Theorem \ref{thm2} shows the central limit theorem now takes the required form.

\vskip 6pt
$(d)$ Noting that all integrands in \eqref{eqn7} are non-negative, the three equalities will together imply that $\log_{T^*}(T_1)=\log^s_{T^*}(T_1)$ a.s., so that for $i=\alpha,\beta,\gamma$ 
\begin{eqnarray}
\left\langle\sum_{i=1}^n\left\{F_i\circ\Phi_{T^*}(T_i)-T^*\right\},\,\e_i\right\rangle=0\quad\hbox{ a.s..}
\label{eqn19}
\end{eqnarray}

On the other hand, if it were possible that, for arbitrarily large $n$, $\hat T_n$ lies in one of the $\alpha$- $\beta$- or $\gamma$-strata, say the $\alpha$-stratum, then $\langle \hat T_n-T^*,\e_\alpha\rangle>0$. On the other hand, since
\begin{eqnarray*}
\hat T_n-T^*&=&\frac{1}{n}\sum_{i=1}^n\left\{\Phi_{\hat T_n}(T_i)-T^*\right\}\\
&\approx&\frac{1}{n}\sum_{i=1}^n\left\{F_\alpha\circ\Phi_{T^*}(T_i)-T^*\right\}+\frac{1}{n}(\hat T_n-T^*)\sum_{i=1}^nM_{T^*}(T_i),
\end{eqnarray*}
and since, as noted in $(b)$, the first diagonal element of $M_{T^*}(T_i)$ is non-positive and the remaining entries in the first row and column of $M_{T^*}(T_i)$ are all zero, we have by \eqref{eqn19} that
\[\langle\hat T_n-T^*,\,\e_\alpha\rangle\approx\left\langle\frac{1}{n}(\hat T_n-T^*)\sum_{i=1}^nM_{T^*}(T_i),\,\e_\alpha\right\rangle\leqslant0.\]
This contradiction implies that, for all sufficiently large $n$, $\hat T_n$ must lie in the stratum $\O(\Sigma)$ of co-dimension one. Thus, the argument for ({\it a}) implies that, when the inequalities in \eqref{eqn7} are all equalities, the central limit theorem for the sample Fr\'echet means takes the same form as that when the three inequalities are all strict.
\end{proof} 

Similar to the note at the end of the previous section, one can also consider the distribution $\mu'$, induced by $\log_{T^*}$ from $\mu$ on the tangent book of $\Tm$ at $T^*$. Then, one can apply the result of \cite{HHLMMMNOPS} to obtain the limiting distribution of the sample Fr\'echet means of $\mu'$. Again, although the limiting distribution obtained in this way retains the local topological feature of the space, the influence of the global topological structure is lost. More importantly, since there is no clear relationship between the sample Fr\'echet means of $\mu$ and $\mu'$, the limiting distribution for the former cannot be easily deduced from that for the latter.

\section{Strata of higher co-dimension}

The structure of tree space in the neighbourhood of a stratum of higher co-dimension is basically similar to, but in detail rather more complex than, that of a stratum of co-dimension one. For example, a stratum $\sigma$ of co-dimension $l$, where $2\leqslant l\leqslant m$, corresponds to a set of $m-l$ mutually compatible edge-types. It arises as a boundary $(m-l)$-dimensional face of a stratum $\tau$ of co-dimension $l'$, where $0\leqslant l'<l$ and when the internal edges of the trees in $\sigma$ are a particular subset of $m-l$ of the internal edges of the trees in $\tau$. For this situation, we say that $\sigma$ bounds $\tau$ and $\tau$ co-bounds $\sigma$. 

Recall from the previous section that the tangent cone to $\Tm$ at a tree $T$ in $\sigma$ consists of all initial tangent vectors to smooth curves starting from $T$, the smoothness only being one-sided at $T$. For simplicity assume, without loss of generality, that under the isometric embedding of $\Tm$ in $\mathbb{R}^M$ all trees in $\sigma$ have zero for their first $l$ coordinates. Then the tangent cone at $T$ has a stratification analogous to that of $\Tm$ itself in the neighbourhood of $T$: for each stratum $\tau$ of co-dimension $l'$ that co-bounds $\sigma$ in $\Tm$ there is a stratum  $\left(\mathbb{R}^{l-l'}_\tau\right)_+\times\mathbb{R}^{m-l}$ in the tangent cone at $T$, which may be identified with a subset of the full tangent space of $\mathbb{R}^M$ at $T$, where $\mathbb{R}^{m-l}$ is the (full) tangent subspace to $\sigma$ at $T$ and $(\mathbb{R}^{l-l'}_\tau)_+$ is the orthant determined by the edge-types that have positive length in $\tau$ but zero length in $\sigma$. For example, the cone point in $\TT$ is a stratum of co-dimension two. Its tangent cone can be identified with $\TT$ itself. This rather involved structure of the tangent cone results in a much more complicated description of the log map and, consequently, of its behaviour. Nevertheless, with the above conventions, it is possible to generalise our expression for the log map to this wider context and hence to obtain analogues of Theorem \ref{thm1} as well as Lemma \ref{lem2}. These results can then be used to describe, in a fashion similar to those of Lemmas \ref{lem3} and \ref{lem4}, certain limiting behaviour of sample Fr\'echet means when their limit lies in a stratum of higher co-dimension. For example, the limiting behaviour of sample Fr\'echet means in $\TT$ when the true Fr\'echet mean lies at the cone point has been studied in \cite{BLO}. The picture given there is incomplete and, although those results can be further refined and improved, it is clear that a complete description of the limiting behaviour of sample Fr\'echet means in the wider context is still a challenge and the global topological structure of the space will play a crucial role.   

\vskip 10pt
\textit{Acknowledgements}. The second author acknowledges funding from the Engineering and Physical Sciences Research Council. The third author acknowledges the support of the Fields Institute.

\end{document}